\documentclass[11pt]{amsart}
\usepackage{graphicx,amsfonts}
\usepackage{amssymb}
\usepackage{amscd}
\usepackage{amsmath}
\usepackage[all]{xy}
\usepackage{mathrsfs}
\usepackage{graphicx}
 \usepackage{color}

\theoremstyle{plain}
\newtheorem{theorem}{Theorem}[section]
\newtheorem{lemma}[theorem]{Lemma}

\newtheorem{proposition}[theorem]{Proposition}

\newtheorem{conjecture}[theorem]{Conjecture}

\theoremstyle{definition}
\newtheorem{example}[theorem]{Example}

\newtheorem{definition}[theorem]{Definition}
\newtheorem{remark}[theorem]{Remark}

  1

\DeclareMathOperator{\Gal}{Gal}

\DeclareMathOperator{\Hom}{Hom}

\DeclareMathOperator{\Spec}{Spec}
\DeclareMathOperator{\N}{N}

\DeclareMathOperator{\im}{im}

\newcommand{\CC}{\mathbb{C}}

\newcommand{\NN}{\mathbb{N}}
\newcommand{\QQ}{\mathbb{Q}}

\newcommand{\RR}{\mathbb{R}}

\newcommand{\ZZ}{\mathbb{Z}}
\newcommand{\Z}{\mathbb{Z}}

\newcommand{\cA}{\mathcal{A}}

\newcommand{\calF}{\mathcal{F}}
\newcommand{\cH}{\mathcal{H}}
\newcommand{\cI}{\mathcal{I}}
\newcommand{\cP}{\mathcal{P}}

\newcommand{\cR}{\mathcal{R}}
\newcommand{\cS}{\mathcal{S}}
\newcommand{\co}{\mathcal{O}}

\newcommand{\cG}{\mathcal{G}}
\newcommand{\cK}{\mathcal{K}}

\newcommand{\cO}{\mathcal{O}}

\newcommand{\fq}{\mathfrak{q}}
\newcommand{\fn}{\mathfrak{n}}
\newcommand{\fm}{\mathfrak{m}}

\newcommand{\fQ}{\mathfrak{Q}}

\newcommand{\cN}{\mathcal{N}}

\def\bigcapp{\raise1ex\hbox{\rotatebox{180}{$\biguplus$}}}
\def\bigcappd{\raise1ex\hbox{\rotatebox{180}{$\displaystyle\biguplus$}}}

\newcommand{\ord}{\mathrm{ord}}
\newcommand{\ram}{\mathrm{ram}}
\newcommand{\tor}{\mathrm{tor}}
\newcommand{\tf}{\mathrm{tf}}
\newcommand{\ab}{\mathrm{ab}}
\newcommand{\GES}{\mathrm{RES}}
\newcommand{\ES}{\mathrm{ES}}

\newcommand{\VS}{\mathrm{VS}}
\newcommand{\Pmod}[1]{\ (\mathrm{mod}\ #1)}
\newcommand{\Fr}{\mathrm{Fr}}
\newcommand{\catname}[1]{\textnormal{{\textsf{#1}}}}
\newcommand{\Det}{\catname{d}}
\newcommand{\bidual}{\bigcap\nolimits}
\newcommand{\exprod}{\bigwedge\nolimits}

\newcommand{\DR}{\catname{R}}
\newcommand{\DL}{\catname{L}}

\begin{document}

\title[]{On Euler systems\\ for the multiplicative group \\ over general number fields}

\author{David Burns, Alexandre Daoud, Takamichi Sano and Soogil Seo}

\begin{abstract} We formulate, and provide strong evidence for, a natural generalization of a conjecture of Robert Coleman concerning Euler systems for the multiplicative group over arbitrary number fields.
\end{abstract}

\address{King's College London,
Department of Mathematics,
London WC2R 2LS,
U.K.}
\email{david.burns@kcl.ac.uk}

\address{King's College London,
Department of Mathematics,
London WC2R 2LS,
U.K.}
\email{alexandre.daoud@kcl.ac.uk}

\address{Osaka City University,
Department of Mathematics,
3-3-138 Sugimoto\\Sumiyoshi-ku\\Osaka\\558-8585,
Japan}
\email{sano@sci.osaka-cu.ac.jp}

\address{Yonsei University,
Department of Mathematics,
Seoul,
Korea.}
\email{sgseo@yonsei.ac.kr}

\thanks{
\hskip 0.155truein Preliminary version of May 2019.}

\maketitle


\section{Introduction}\label{intro}

\subsection{}Ever since its introduction by Kolyvagin in the late 1980s, the theory of Euler systems has played a vital role in the proof of many celebrated results concerning the structure of Selmer groups attached to $p$-adic representations that are defined over number fields and satisfy a variety of technical conditions.

In order to consider a wider class of representations, the theory has also been developed in recent years to incorporate a natural notion of `higher rank' Euler systems.

Given their importance, it is clearly of interest to understand the full collection of Euler systems (of the appropriate rank) that arise in any given setting.

In the concrete setting of (rank one) Euler systems that are attached to the multiplicative group $\mathbb{G}_m$ over abelian extensions of $\QQ$, a conjecture of Coleman concerning circular distributions can be seen to imply that all such systems should arise in an elementary way from the classical theory of cyclotomic units.

This conjecture of Coleman was itself motivated by an `archimedean characterization' of norm-coherent sequences of cyclotomic units that he had obtained in \cite{coleman2} and hence to attempts to understand a globalized version of the fact that all norm-compatible families of units in towers of local cyclotomic fields arise by evaluating a `Coleman power series' at roots of unity.

To consider the analogous problem for an arbitrary number field $K$ we write $r_K$ for the number of archimedean places of $K$ and $K^{\rm s}$ for the maximal abelian extension of $K$ in which all archimedean places split completely.

Then, by adapting a general construction for $p$-adic representations that is described by the first and third authors in \cite{sbA}, we shall first define (unconditionally) a module ${\rm ES}^{\rm b}_K$ of `basic' Euler systems of rank $r_K$ for $\mathbb{G}_m$ relative to the extension $K^{\rm s}/K$.

In the case $K = \QQ$ we shall prove in Theorem \ref{etnc prop} that ${\rm ES}^{\rm b}_K$ is generated by the restriction of the cyclotomic Euler system to real abelian fields.

In the general case, we shall show in Theorem \ref{tech req} that ${\rm ES}^{\rm b}_K$ contains systems that are directly related to the leading terms of Artin $L$-series at $s=0$.

We shall then predict that essentially all Euler systems of rank $r_K$ for $\mathbb{G}_m$ over $K^{\rm s}/K$ should belong to ${\rm ES}^{\rm b}_K$ (for a precise statement see Conjecture \ref{main conj}).

In the case $K = \QQ$ we can deduce from recent results of the first and fourth authors in \cite{bs} that Conjecture \ref{main conj} is equivalent to Coleman's original conjecture on circular distributions, and hence that the results of loc. cit. give strong evidence for our conjecture in this case.

In addition, to obtain evidence for the general case of Conjecture \ref{main conj} we can incorporate the construction of basic Euler systems into the equivariant theory of higher rank Euler, Kolyvagin and Stark systems for $\mathbb{G}_m$ that is developed by Sakamoto and the first and third authors in \cite{bss3}.

In particular, in this way we shall prove in Theorem \ref{koly pre thm} (and Theorem \ref{koly thm}) that the main result of \cite{bss3} leads to some strong, and unconditional, evidence in support of Conjecture \ref{main conj}.

In fact, we find that even in the case $K = \QQ$ the latter result constitutes a strong improvement of results in the literature.

For example, in Theorem \ref{nc thm}, we shall use it to prove a natural algebraic analogue of Coleman's `archimedean characterization' of norm coherent sequences in towers of cyclotomic fields (which is the main result of \cite{coleman2}). 

\subsection{Notation} In the remainder of the introduction we shall, for the convenience of the reader, collect together various notations that we employ throughout this article.

\subsubsection{Arithmetic}
We fix an algebraic closure $\QQ^c$ of $\QQ$ and an algebraic closure $\QQ_p^c$ of $\QQ_p$ for each prime number $p$.

Throughout the article, $K$ will denote a number field (that is, a finite extension of $\QQ$ in $\QQ^c$) and $K^{\ab}$ the maximal abelian extension of $K$ inside $\QQ^c$. We write $S_\infty(K)$ for the set of archimedean places of $K$ and denote its cardinality by $r_K$.

For each extension $\mathcal{K}$ of $K$ in $K^{\rm ab}$ we consider the following collection of intermediate fields
\[ \Omega(\mathcal{K}/K) := \{ F \mid  K \subset F \subseteq \mathcal{K}, \,\,F/K \,\text{ is ramified and of finite degree}\}.\]

For each field $E$ in $\Omega(\mathcal{K}/K)$ we set
\begin{align*}
    \mathcal{G}_E := \Gal(E/K)
\end{align*}
and we write $S_\ram(E/K)$ for the (finite, non-empty) set of places of $K$ that ramify in $E$.

For any set $S$ of places of $K$, we then set
$$S(E):=S \cup S_{\rm ram}(E/K).$$

We write $S_E$ for the set of places of $E$ lying above those in $S$. The ring of $S_E$-integers of $E$ is denoted by
\[ \mathcal{O}_{E,S} := \{a \in E\, |\, \ord_w(a) \geq 0 \text{ for all non-archimedean } w \not\in S_E\}
\]
where $\ord_w$ is the normalised additive valuation relative to $w$.

Given a finite set $T$ of places of $K$ that is disjoint from $S\cup S_\infty(K)$, we define the $(S,T)$-unit group of $E$ to be the finite index subgroup of $\mathcal{O}_{E,S}^\times$ given by
\begin{align*}
    \mathcal{O}_{E,S,T}^\times := \{a \in \mathcal{O}_{E,S}^\times \,|\, a \equiv 1 \Pmod{w} \text{ for all } w \in T_E\}.
\end{align*}

We denote by $Y_{E,S}$ the free $\mathbb{Z}$-module on $S_E$ and $X_{E,S}$ its augmentation kernel so that there is a tautological short exact sequence
\begin{equation*}\label{augmentation-SES}
  0 \to X_{E,S} \to Y_{E,S} \to \mathbb{Z} \to 0.
\end{equation*}

Given a place $w$ of $E$, we denote by $\mathcal{G}_{E,w}$ the decomposition subgroup of $\cG_E$ relative to $w$. If $w$ is non-archimedean, we denote by $\kappa(w)$ the residue field relative to $w$ and write $\mathrm{N}w$ for the cardinality of $\kappa(w)$. Let $v$ be the place of $K$ lying under $w$. If $v$ does not ramify in $E$ then we write $\Fr_v \in \mathcal{G}_{E,w}$ for the Frobenius automorphism relative to $w$.

We denote by $\mathcal{R}_{\mathcal{K}/K}$ the completed group ring $\mathbb{Z}[[\Gal(\mathcal{K}/K)]]$. For each $E$ in $\Omega(\mathcal{K}/K)$ we write $\mu_E$ for the torsion subgroup of $E^\times$ and $\mathcal{A}_E$ for the annihilator in $\ZZ[\cG_E]$ of $\mu_E$. Then for each $E'$ with $E \subseteq E'$ one has $\pi_{E'/E}(\mathcal{A}_{E'}) =\mathcal{A}_E$ (see Lemma \ref{torsion lemma}(ii)) and so we obtain an ideal of $\mathcal{R}_{\mathcal{K}/K}$ by setting
\[ \mathcal{A}_{\mathcal{K}/K} := \varprojlim_{E \in \Omega(\cK/K)}\mathcal{A}_E.\]

We write $K^{\rm s}$ for the maximal subfield of $K^{\rm ab}$ in which all archimedean places of $K$ split. We then abbreviate the modules $\mathcal{R}_{K^{\rm ab}/K}$, $\mathcal{R}_{K^{\rm s}/K}$, $\mathcal{A}_{K^{\rm ab}/K}$ and $\mathcal{A}_{K^{\rm s}/K}$ to $\mathcal{R}_K$, $\mathcal{R}^{\rm s}_K, \mathcal{A}_{K}$ and $\mathcal{A}^{\rm s}_{K}$ respectively.

For each natural number $m$ we write $\mu_m$ for the group of $m$-th roots of unity in $\QQ^c$. For each $n$ we fix a generator $\zeta_n$ of $\mu_n$ such that $\zeta_{mn}^m = \zeta_n$ for all $m$ and $n$.

\subsubsection{Algebra}

Let $R$ be a commutative noetherian ring and $X$ an $R$-module. We set $X^* := \Hom_R(X, R)$. 
 We write $X_\tor$ for the torsion subgroup of $X$ and $X_\tf$ for the associated torsion-free quotient $X/X_\tor$. We also write $\widehat{X}$ for the profinite completion of $X$ and $\widehat{X}^p$ for its pro-$p$ completion at a rational prime $p$.

For a non-negative integer $r$ the $r$-th exterior power bidual of $X$ is defined by setting
\begin{align*}
     \bidual_{R}^r X := \left(\exprod_R^r (X^*)\right)^*.
\end{align*}
We recall that there exists a natural homomorphism of $R$-modules
$$\xi_X^r:{\bigwedge}_R^r X \to \bidual_{R}^r X; \ x \mapsto (\Phi \mapsto \Phi(x))$$
that is, in general, neither injective nor surjective (see \cite[(1)]{sbA}).

Exterior power biduals play an essential role in the theory of higher rank Euler, Kolyvagin and Stark systems (see, for example, \cite{bss3} and \cite{sbA}).

In particular, if $R = \mathbb{Z}[G]$ for a finite abelian group $G$, then $\xi_X^r$ induces an identification 
\begin{equation}\label{bidual-lattice-identification}
  \left\{x \in \exprod_{\QQ[G]}^r (\QQ \otimes_\ZZ X) \,\middle|\, \Phi(x) \in R \text{ for all } \Phi \in \exprod_{R}^r (X^*)\right\}   \xrightarrow{\sim}  \bidual_R^r X \end{equation}
(cf. \cite[Prop. A.8]{sbA}). Lattices of this form were first used by Rubin in \cite{rubin} to formulate an integral refinement of Stark's Conjecture.

In order to ease notations for objects associated to group rings, we shall sometimes abbreviate a subscript `$\mathbb{Z}[G]$' to `$G$' (for example, writing `${\bigcap}_{G}^r$' rather than `${\bigcap}_{\ZZ[G]}^r$').

We denote the cardinality of a finite set $X$ by $|X|$.

\section{Statement of the conjecture and main results}\label{basic prop sec}

\subsection{Euler systems}We first introduce the various notions of Euler system with which we shall be concerned.

In this section, we set $S:=S_\infty(K)$ so that $S(E)=S_\infty(K)\cup S_{\rm ram}(E/K)$ for any finite extension $E/K$.

Fix a field $\mathcal{K}$ in $K^{\ab}$, an integer $r \geq 0$ and a pair of fields $E$ and $E'$ in $\Omega(\mathcal{K}/K)$ with $E \subseteq E'$. Then the field-theoretic norm map $\N_{E'/E}: \QQ \otimes_{\mathbb{Z}} \mathcal{O}_{E',S(E')}^\times \to \QQ \otimes_{\mathbb{Z}} \mathcal{O}_{E,S(E')}^\times$ induces a commutative diagram

$$
\xymatrix{\displaystyle{\bidual_{\mathcal{G}_{E'}}^r \mathcal{O}_{E',S(E')}^\times} \ar[d]^{{\rm N}^r_{E'/E}} \ar@^{^{(}->}[r] & \displaystyle{\QQ \otimes_\ZZ \exprod_{\mathcal{G}_{E'}}^r \mathcal{O}_{E',S(E')}^\times} \ar[d]^{{\rm N}^r_{E'/E}}\\
\displaystyle{\bidual_{\mathcal{G}_{E}}^r \mathcal{O}_{E,S(E')}^\times} \ar@^{^{(}->}[r] & \displaystyle{\QQ \otimes_\ZZ \exprod_{\mathcal{G}_{E}}^r \mathcal{O}_{E,S(E')}^\times}}
$$
where the horizontal inclusions are induced by the identification (\ref{bidual-lattice-identification}).


\begin{definition}
    A `rational Euler system of rank $r$' for $\mathbb{G}_m$ and $\mathcal{K}/K$ is a collection
    \begin{align*}
        c=(c_E)_{E  } \in \prod_{E \in \Omega(\mathcal{K}/K)} \left(\QQ \otimes_\ZZ {\bigwedge}_{\cG_E}^r \cO_{E,S(E)}^\times\right)
    \end{align*}
that satisfies the following distribution relations: for every pair of fields $E$ and $E'$ in $\Omega(\mathcal{K}/K)$ with $E\subseteq E'$ one has
\begin{equation}\label{dist rel} {\N}_{E'/E}^r(c_{E'})= \left( \prod_{v \in S(E')\setminus S(E)} (1-{\rm Fr}_v^{-1})\right)c_E \end{equation}
in $\QQ \otimes_\ZZ {\bigwedge}_{\cG_E}^r \cO_{E,S(E')}^\times$.

The set ${\GES}_r(\mathcal{K}/K)$ of all rational Euler systems of rank $r$ has a natural structure as an $\mathcal{R}_{\cK/K}$-module. If $r = r_K$ and $\mathcal{K} = K^{\rm s}$, then we abbreviate ${\GES}_r(\cK/K)$ to
${\GES}_K$.
\end{definition}

\begin{example} A classical example of a rational Euler system is given by the system of cyclotomic units. In fact, if for each $E$ in $\Omega(\mathbb{Q}^{\rm s}/\mathbb{Q})$ of (non-trivial) conductor $\mathfrak{f}(E)$, so that $E \subseteq \mathbb{Q}(\zeta_{\mathfrak{f}(E)}),$ one sets $c_E := \N_{\mathbb{Q}(\zeta_{\mathfrak{f}(E)})/E}(1-\zeta_{\mathfrak{f}(E)})$, then the collection $(c_E)_E$ belongs to ${\rm RES}_\QQ$.

More generally, the Euler system of (conjectural) Rubin-Stark elements constitutes an element of $\GES_K$. In particular, for each field $E$ in $\Omega(K^{\rm s}/K)$, the (`$T$-less') Rubin-Stark conjecture for the data set $(E/K, S(E), \varnothing, S_\infty(K))$ is formulated in \cite[Conj. A]{rubin} and predicts the existence of an element $c_E$ in $\QQ \otimes_\ZZ \exprod_{\mathcal{G}_E}^r \mathcal{O}_{E,S(E)}^\times$ that is related by the Dirichlet regulator isomorphism to the $r$-th order $S(E)$-truncated Stickelberger element of $E/K$. Then the argument of \cite[Prop. 6.1]{rubin} shows that $(c_E)_E$ belongs to $\GES_K$.
\end{example}

\begin{definition}\label{def es} `An Euler system of rank $r$' for the pair $(\mathbb{G}_m, \mathcal{K}/K)$ is an element $c$ of ${\GES}_r(\mathcal{K}/K)$ with the property that for every field $E$ in $\Omega(\cK/K)$ the element $c_E$ belongs to the lattice ${\bidual}_{\cG_E}^r \cO_{E,S(E)}^\times$.

The collection ${\ES}_r(\mathcal{K}/K)$ of such systems is an $\cR_{K}$-submodule of
${\GES}_r(\mathcal{K}/K)$. If $r = r_K$ and $\mathcal{K} = K^{\rm s}$, then we abbreviate ${\ES}_r(\mathcal{K}/K)$ to
${\ES}_K$.
\end{definition}

\begin{remark}\label{not classical}The module ${\rm ES}^{\rm cl} := {\rm ES}^{\rm cl}(\mathcal{K}/K)$ of `classical' (rank one) Euler systems' for $\mathbb{G}_m$ and $\mathcal{K}/K$ comprises all elements $c = (c_E)_E$ of $\prod_E\cO_{E,S(E)}^\times$, where $E$ runs over $\Omega(\mathcal{K}/K)$, with the property that for all pairs $E$ and $E'$ with $E \subseteq E'$ the distribution relation (\ref{dist rel}) is valid in the group $\cO_{E,S(E')}^\times$.

In particular, since for every $E$ the module $\bidual_{\mathcal{G}_E}^1 \mathcal{O}_{E,S(E)}^\times := (\mathcal{O}_{E,S(E)}^\times)^{**}$ identifies with $\mathcal{O}^\times_{E,S(E), \tf}$ the natural reduction map from ${\rm ES}^{\rm cl}$ to the module ${\rm ES}_1 := {\rm ES}_1(\mathcal{K}/K)$ need be neither injective nor surjective.

However, if $c$ belongs to ${\rm ES}_1$, then for any element $a$ of $\mathcal{A}_{\mathcal{K}/K}$ there exists a canonical system $c' = c'_{a}$ in  ${\rm ES}^{\rm cl}$ that projects to give $c^a$ in ${\rm ES}_1$: for each $E$ in $\Omega(\mathcal{K}/K)$ one need only define $c'_E$ to be $(\tilde c_E)^{a_E}$ where $\tilde c_E$ is any choice of element of $\cO_{E,S(E)}^\times$ that projects to $c_E$ in ${\bigcap}^1_{\cG_E}\cO_{E,S(E)}^\times$.
\end{remark}

\subsection{The conjecture and main results}In \S\ref{basic es section} below we will use certain natural families of \'etale cohomology complexes to construct a canonical `invertible' $\mathcal{R}^{\rm s}_K$-module ${\rm VS}(K^{\rm s}/K)$ and to prove the existence of a canonical non-zero homomorphism
\[ \Theta^{\rm s}_K: {\rm VS}(K^{\rm s}/K) \to {\rm RES}_K\]
%
of $\mathcal{R}_K$-modules with the property that
\begin{equation}\label{1st inclusion} \mathcal{A}^{\rm s}_K\cdot \im(\Theta_K^{\rm s}) \subseteq {\rm ES}_K.\end{equation}

The constructions of the module ${\rm VS}(K^{\rm s}/K)$ and map $\Theta_K^{\rm s}$ arise by adapting certain generic $p$-adic constructions from \cite{sbA} that are unconditional and both essentially algebraic, and quite elementary, in nature.

As a result, the inclusion (\ref{1st inclusion}) implies that for every field $K$ there exists a `large' module of Euler systems of rank $r_K$ (for details see Theorem \ref{tech req}).

We define the module of `basic Euler systems of rank $r_K$' for $\mathbb{G}_m$ over $K$ by setting
\[ {\rm ES}^{\rm b}_K := \im(\Theta_K^{\rm s}).\]

In this article we shall then study the following conjecture concerning a partial converse to the inclusion (\ref{1st inclusion}).

\begin{conjecture}\label{main conj} $\mathcal{A}^{\rm s}_K\cdot {\rm ES}_K \subseteq \mathcal{A}^{\rm s}_K\cdot {\rm ES}^{\rm b}_K$. \end{conjecture}

This conjecture asserts that, modulo minor technical issues concerning torsion, {\em all} Euler systems in ${\rm ES}_K$ should arise via the elementary construction given in \S\ref{vds sec}.

\begin{remark} It can be shown that ${\rm ES}^{\rm b}_\QQ$ is not contained in ${\rm ES}_\QQ$ (see Lemma \ref{tor ann lem}). However, if we set $\Lambda := \ZZ\left[\frac{1}{2}\right]$, then for any totally real field $K$ one has $\Lambda\otimes_\ZZ\mathcal{A}^{\rm s}_K = \Lambda\otimes_\ZZ\mathcal{R}^{\rm s}_{K}$ and so Conjecture \ref{main conj} combines with (\ref{1st inclusion}) to predict $\Lambda\otimes_\ZZ{\rm ES}_K = \Lambda\otimes_\ZZ{\rm ES}^{\rm b}_K$. In fact, at this stage, we do not know an example showing that the inclusion of Conjecture \ref{main conj} should not itself always be an equality.\end{remark}

To obtain evidence for Conjecture \ref{main conj} one can use the equivariant theory of higher rank Kolyvagin systems for $\mathbb{G}_m$, as developed by Sakamoto and the first and third authors in \cite{bss3}. In particular, in this way we shall obtain the following (unconditional) result, a precise statement, and proof, of which is given in \S\ref{koly section}.

\begin{theorem}[{Theorem \ref{koly thm}}]\label{koly pre thm} Fix an odd prime $p$ and for each field $E$ in $\Omega(K^{\rm s}/K)$ write $\Delta_E$ for the maximal subgroup of  $\cG_E$ of order prime to $p$. Then for every system $c$ in ${\rm ES}_K$, every field $E$ in $\Omega(K^{\rm s}/K)$ and every homomorphism $\chi: \Delta_E \to \QQ_p^{c,\times}$ that satisfies certain mild technical hypotheses, the `$\chi$-component' of $c_E$ belongs to the $\ZZ_p[\cG_E]$-module generated by  $\{c'_E \mid c'\in {\rm ES}^{\rm b}_K\}$.\end{theorem}

If $K$ has at least one real embedding, then $\mu_{L} = \{\pm 1\}$ for all $L$ in $\Omega(K^{\rm s}/K)$ and so $\mathcal{A}_K^{\rm s}$ is the kernel $\mathcal{I}_{K,(2)}$ of the natural `mod $2$ augmentation' map $\mathcal{R}^{\rm s}_K \to \ZZ/(2)$.

In particular, in this case Conjecture \ref{main conj} predicts that every Euler system in ${\rm ES}_K$ has the form $2^{-1}\cdot \Theta_K^{\rm s}(x)$ for some $x$ in $\mathcal{I}_{K,(2)}\cdot {\rm VS}(K^{\rm s}/K)$.

However, as the following result shows, in special cases it can predict much more.

\begin{theorem}\label{Q theorem} If $K = \QQ$, then Conjecture \ref{main conj} is equivalent to Coleman's Conjecture on circular distributions. In particular, if Conjecture \ref{main conj} is valid, then for every $c = (c_F)_F$ in ${\rm ES}^{\rm cl}(\QQ^{\rm ab}/\QQ)$ there exists an element $r_c$ of $\mathcal{R}_\QQ$ such that
\[ c_{\QQ(n)} = \pm (1-\zeta_n)^{r_c}\]
for every $n \not\equiv 2$ (mod $4$).
\end{theorem}

This result shows that Conjecture \ref{main conj} constitutes a natural `generalized Coleman Conjecture' and will be proved in \S\ref{Q thm sec}.

In addition, by combining Theorem \ref{koly pre thm} with the analysis used to prove Theorem \ref{Q theorem} we obtain the following algebraic analogue of the `archimedean characterization' of norm coherent sequences in towers of the form $\bigcup_n\QQ(\mu_{p^n})$ that was given by Coleman in \cite{coleman2}. (We recall that this result was the original motivation for Coleman's study of circular distributions.)

\begin{theorem}\label{nc thm} Let $p$ be an odd prime and $(a_n)_n$ a norm coherent sequence in the tower of fields $\bigcup_n\QQ(\mu_{p^n})$. Then each element $a_n$ belongs to the $\mathcal{R}_\QQ$-module generated by $1-\zeta_{p^n}$ if and only if there exists a circular distribution $f$ and a non-negative integer $t$ such that $f(\zeta_{p^n})^{2^t} = a_{n}$ for all $n$.
\end{theorem}

This result improves upon \cite[Th. B]{Seo2} (by showing that the positive integer $c$ in loc. cit. can be taken as a power of $2$) and also gives an affirmative answer, modulo powers of $2$, to the question raised by the third author in \cite[end of \S1]{Seo4}. It will be proved in \S\ref{nc sec}.

\section{Basic Euler systems}\label{basic es section}

In this section we shall define, and establish the key properties of, the homomorphism $\Theta_K^{\rm s}$ that occurs in Conjecture \ref{main conj}.

To do this we fix an extension $\mathcal{K}$ in $K^\ab$ and a finite set of places $S$ of $K$ containing $S_\infty(K)$. (Later in the article we will specialise to $S=S_\infty(K)$.)

Throughout this section we shall also use the following general notation. For any finite abelian group $\mathcal{G}$ and any commutative noetherian ring $R$ we write $D(R[\mathcal{G}])$ for the derived category of $R[\mathcal{G}]$-modules and $D^{\rm perf}(R[\mathcal{G}])$ for the full triangulated subcategory of $D(R[\mathcal{G}])$ comprising complexes that are perfect.

For a bounded above complex of $\mathcal{G}$-modules $C$ we write $C^\ast$ for the object $\DR\Hom_\Z(C,\Z[0])$ of $D(\ZZ[\mathcal{G}])$, where the linear dual is endowed with the natural contragredient action of $\mathcal{G}$.

We write ${\Det}_{R[\cG]}(-)$ for the determinant functor on $D^{\rm perf}(R[\mathcal{G}])$, as constructed by Knudsen and Mumford in \cite{KM}.

We write $\widehat{\mathcal{G}}$ for the group of $\QQ^c$-valued characters of $\mathcal{G}$ and for any such $\chi$ we write $e_\chi$ for idempotent $|\mathcal{G}|^{-1}\sum_{\sigma \in \mathcal{G}} \chi(\sigma)\sigma^{-1}$ of $\QQ^c[\mathcal{G}]$.

For any $\mathcal{G}$-module $M$ we also set $M^\chi := e_\chi\cdot(\QQ^c\cdot M)$, where $\QQ^c\cdot M$ denotes the $\QQ^c[\mathcal{G}]$-module that is generated by $M$.

\subsection{Modified \'etale cohomology complexes}\label{vds sec}

For each $E$ in $\Omega(\mathcal{K}/K)$ and each finite set of places $T$ of $K$ that is disjoint from $S(E)$ the methods of Kurihara and the first and third authors in \cite[\S2]{bks1} define a canonical $T$-modified, compactly supported `Weil-\'etale' cohomology complex $\DR\Gamma_{c,T}((\mathcal{O}_{E,S(E)})_\mathcal{W}, \ZZ)$ of the constant sheaf $\ZZ$ on ${\rm Spec}(\mathcal{O}_{E,S(E)})$.

In the sequel we set
\[ C_{E,S(E),T} := \DR\Gamma_{c,T}((\mathcal{O}_{E,S(E)})_\mathcal{W}, \ZZ)^\ast[-2]\]
and shall use the properties of this complex that are recalled in the following result.

For any finite set $T'$ of places of $K$ that is disjoint from $S(E)$ we set
\[ {\mathbb F}_{T'_E}^\times:= \bigoplus_{w \in T'_{E}} \kappa(w)^{\times}.\]

For each non-archimedean place $v$ of $K$ we write $\DR\Gamma(\kappa(v)_\mathcal{W}, \ZZ[\cG_E])$ for the direct sum over places $w$ of $E$ above $v$ of the complexes $\DR\Gamma(\kappa(w)_\mathcal{W}, \ZZ)$ defined in \cite[Prop. 2.4(ii)]{bks1}.


For any finite set of places $\Sigma$ of $K$ that contains $S_\infty(K)$ and is disjoint from $T$, we write ${\rm Cl}_{\Sigma}^T(E)$ for the ray class group of $\mathcal{O}_{E,\Sigma}$ modulo the product of all places of $E$ above $T$.

\begin{proposition}\label{new one} For each $E$ in $\Omega(\cK/K)$ the complex $C_{E,S(E),T}$ belongs to $D^{\rm perf}(\ZZ[\cG_E])$ and has all of the following properties.
\begin{itemize}
\item[(i)] For any finite set of places $\Sigma$ of $K$ that contains $S(E)$ and is disjoint from $T$ the complex $C_{E,\Sigma,T}$ is acyclic outside degrees zero and one and there are canonical identifications of $\cG_E$-modules $H^0(C_{E,\Sigma,T}) = \mathcal{O}_{E,\Sigma,T}^\times$, $H^1(C_{E,\Sigma,T})_{\rm tor} = {\rm Cl}_{\Sigma}^T(E)$ and $H^1(C_{E,\Sigma,T})_{\rm tf} = X_{E,\Sigma}$.
\item[(ii)] If $T'$ is any finite set of places that contains $T$ and is disjoint from $S(E)$, then there is a canonical exact triangle in $D^{\rm perf}(\ZZ[\cG_E])$
\[ C_{E,S(E),T'} \to C_{E,S(E),T} \to {\mathbb F}_{(T'\setminus T)_E}^\times[0] \to . \]
\item[(iii)] Given a finite set $S'$ of places of $K$ that contains $S$ and is disjoint from $T$, there exists a canonical exact triangle in
$D^{\rm perf}(\ZZ[\cG_E])$ of the form
\[ C_{E,S(E),T} \to C_{E,S'(E),T} \to \bigoplus_{v \in S'(E)\setminus S(E)} \DR\Gamma(\kappa(v)_\mathcal{W}, \ZZ[\cG_E])^\ast[-1]\to .\]
\item[(iv)]  For any fields $E$ and $E'$ in $\Omega(\mathcal{K}/K)$ with $E \subseteq E'$ there exists a natural isomorphism
\[ \ZZ[\cG_E]\otimes_{\ZZ[\cG_{E'}]}^{\DL}C_{E',S(E'),T} \cong C_{E,S(E'),T}\]
in $D^{\rm perf}(\ZZ[\cG_E])$.
\end{itemize}
\end{proposition}

\begin{proof} Write $D^-(\ZZ[\cG_E])$ for the full subcategory of $D(\ZZ[\cG_E])$ comprising complex that are cohomologically bounded above and $\psi^\ast$ for the functor from $D^-(\ZZ[\cG_E])$ to $D(\ZZ[\cG_E])$ that sends each complex $C$ to $C^\ast[-2]$.

Concerning claim (i), since $\Sigma$ contains $S(E)$, the fact that $C_{E,\Sigma,T}$ belongs to $D^{\rm perf}(\ZZ[\cG_E])$ is a consequence of \cite[Prop. 2.4(iv)]{bks1} and the fact that $\psi^\ast$ preserves $D^{\rm perf}(\ZZ[\cG_E])$. The descriptions of cohomology given in claim (i) are stated in \cite[Rem. 2.7]{bks1}.

The exact triangles in claim (ii) and (iii) result from applying $\psi^\ast$ to the triangles given by the right hand column of the diagram in \cite[Prop. 2.4(i)]{bks1} and to the exact triangle in \cite[Prop. 2.4(ii)]{bks1} respectively.

The existence of isomorphisms as in claim (iv) can be deduced by combining the commtutative diagram of exact triangles in \cite[Prop. 2.4(i)]{bks1} together with the well-known isomorphisms
\[ \ZZ[\cG_E]\otimes_{\ZZ[\cG_{E'}]}^{\DL}\DR\Gamma_c((\mathcal{O}_{E',S(E')})_{\rm\acute e t},\ZZ)^\ast \cong \DR\Hom_{\ZZ[\cG_{E'}]}(\ZZ[\cG_E],\DR\Gamma_c((\mathcal{O}_{E',S(E')})_{\rm\acute e t},\ZZ))^\ast\]
and
\[ \DR\Hom_{\ZZ[\cG_{E'}]}(\ZZ[\cG_E],\DR\Gamma_c((\mathcal{O}_{E',S(E')})_{\rm\acute e t},\ZZ))\cong \DR\Gamma_c(
(\mathcal{O}_{E,S(E')})_{\rm\acute e t},\ZZ).\]
\end{proof}

\begin{remark}\label{F-T-resolution-remark} For each place $v$ of $K$ outside $S(E)$, claim (ii) of Proposition \ref{new one} with $T= \emptyset$ and $T' = \{v\}$ implies that ${\mathbb F}_{T'_E}^\times[0]$ belongs to $D^{\rm perf}(\ZZ[\cG_E])$. It is in fact straightforward to show (and well-known) that this complex is isomorphic in $D^{\rm perf}(\ZZ[\cG_E])$ to the complex
\[ \mathbb{Z}[\cG_E] \xrightarrow{1 - \mathrm{N}v\cdot\Fr_v^{-1}}\mathbb{Z}[\cG_E]\]
where the first term is placed in degree minus one.
\end{remark}

\begin{remark}\label{extra S resolution remark} For each place $v$ of $K$ outside $S(E)$ the result of \cite[Prop. 3.2]{bf} implies that the complex $\DR\Gamma(\kappa(v)_\mathcal{W}, \ZZ[\cG_E])^\ast[-1]$ that occurs in Proposition \ref{new one}(iii) is canonically isomorphic to the complex
\[ \mathbb{Z}[\cG_E] \xrightarrow{1 - \Fr_v^{-1}} \mathbb{Z}[\cG_E],\]
where the first term is placed in degree zero.\end{remark}

\begin{remark} If $S(E)$ contains all places above a given rational prime $p$, then \cite[Prop. 3.3]{bf} implies that there is a canonical isomorphism in $D^{\rm perf}(\ZZ_p[\mathcal{G}_E])$ of the form
   \[ \mathbb{Z}_p \otimes_{\mathbb{Z}} C_{E,S(E)} \cong \DR\!\Hom_{\ZZ_p}(\DR\Gamma_{c}(\mathcal{O}_{E,S(E)}, \ZZ_p), \ZZ_p)[-2].\]
This isomorphism relates the constructions that we make below to those from \cite{sbA}. \end{remark}

\subsection{Vertical determinantal systems}


In this section we use the complexes discussed in \S\ref{vds sec} to construct a canonical module over the algebra $\mathcal{R}_{\cK/K}$.

We start by proving a useful technical result about the set $\Omega(\cK/K)$.

\begin{lemma}\label{omega-cofinal-tower-lemma}
    There exists a cofinal directed subset of $\Omega(\mathcal{K}/K)$ order-isomorphic to $\NN$.
\end{lemma}

\begin{proof}
   Fix an enumeration $\mathfrak{p}_1, \mathfrak{p}_2, \cdots$ of the places of $K$. Given a modulus $\mathfrak{m}$ of $K$, let $K(\mathfrak{m})$ denote the intersection of $\cK$ with the ray class field modulo $\mathfrak{m}$ for $K$. For each $n \geq 1$, we now set $K_n = K\left(\prod_{i=1}^n\mathfrak{p}_i^n\right)$. By class field theory, each $K_n$ is abelian of finite degree over $K$ and there is a chain of inclusions $K_1 \subseteq K_2 \subseteq \cdots$. Given an extension $L \in \Omega(\mathcal{K}/K)$, let $\mathfrak{f}_L$ be the conductor of $L$. Then $L$ is contained in any $K(\mathfrak{m})$ such that $\mathfrak{f}_L \mid \mathfrak{m}$. Hence we can always choose an $n$ such that $L \subseteq K_n$. We thus see that the tower of fields $\{K_n\}_{n \geq 1}$ form a countable cofinal directed subset of $\Omega(\mathcal{K}/K)$.
\end{proof}

For $E$ in $\Omega(\mathcal{K}/K)$ we abbreviate the functor $\Det_{\ZZ[\mathcal{G}_{E}]}(-)$ to $\Det_{\mathcal{G}_{E}}$. For each pair of fields $E$ and $E'$ in $\Omega(\mathcal{K}/K)$ with $E \subseteq E'$, we then define
\[ \nu_{E'/E}:  \Det_{\mathcal{G}_{E'}}(C_{E',S(E')}) \to  \Det_{\mathcal{G}_E}(C_{E,S(E)})\]
to be the following composite homomorphism of $\ZZ[\cG_{E'}]$-modules

\begin{align*}
    \Det_{\mathcal{G}_{E'}}(C_{E',S(E')}) &\to \ZZ[\cG_E]\otimes_{\ZZ[\cG_{E'}]}\Det_{\mathcal{G}_{E'}}(C_{E',S(E')})\\
    &\xrightarrow{\sim} \Det_{\mathcal{G}_E}(\ZZ[\cG_E]\otimes^\DL_{\ZZ[\cG_{E'}]}C_{E',S(E')})\\
    &\xrightarrow{\sim} \Det_{\mathcal{G}_E}(C_{E,S(E')})\\
    &\xrightarrow{\sim} \Det_{\mathcal{G}_E}(C_{E,S(E)}) \otimes \bigotimes_{v \in S(E')\backslash S(E)} \Det_{\mathcal{G}_E}(\DR\Gamma(\kappa(v)_\mathcal{W}, \ZZ[\cG_E])^\ast[-1])\\
    &\xrightarrow{\sim} \Det_{\mathcal{G}_E}(C_{E,S(E)}).
\end{align*}
Here the first map is the canonical projection, the second is induced by the standard base-change property of determinant functors, the third by the isomorphism in Proposition \ref{new one}(iv) (with $T = \emptyset$), the fourth by the exact triangle in Proposition \ref{new one}(iii) (with $T = \emptyset$), and the last homomorphism is induced by the identification
\begin{equation}\label{triv triv}
  \Det_{\mathcal{G}_E}(\DR\Gamma(\kappa(v)_\mathcal{W}, \ZZ[\cG_E])^\ast[-1]) \cong \Det_{\mathcal{G}_E}(\mathbb{Z}[\mathcal{G}_E]) \otimes_{\mathcal{G}_E} \Det^{-1}_{\mathcal{G}_E}(\mathbb{Z}[\mathcal{G}_E])\cong \mathbb{Z}[\mathcal{G}_E]
\end{equation}
where the first isomorphism is induced by the description of $\DR\Gamma(\kappa(v)_\mathcal{W}, \ZZ[\cG_E])^\ast[-1]$ given in Remark \ref{extra S resolution remark} and the second is the standard `evaluation map' isomorphism.

\begin{definition}\label{definition vertical} The module of `vertical determinantal systems' for $\mathbb{G}_m$ and $\cK/K$ is the $\cR_{\cK/K}$-module given by the inverse limit
\begin{align*}
    \VS(\cK/K) := \varprojlim_{E\in \Omega(\cK/K)} {\Det}_{\cG_E}(C_{E,S(E)})
\end{align*}
where in the inverse limit the transition morphisms are the maps $\nu_{E'/E}$ defined above.
\end{definition}




The following result shows that this $\cR_{\cK/K}$-module is in a natural sense `invertible'.

\begin{proposition}\label{invert prop} For each prime $p$ the pro-$p$-completion of ${\rm VS}(\cK/K)$ is free of rank one over $\widehat{\mathcal{R}}^p_{\cK/K}$. \end{proposition}

\begin{proof} We fix a rational prime $p$ and for $E$ in $\Omega(\cK/K)$ set
\begin{align*}
    \Xi_{E} :=  \Det_{\ZZ_p[\mathcal{G}_E]}(\ZZ_p\otimes_\ZZ C_{E,S(E)})\cong \ZZ_p\otimes_\ZZ\Det_{\mathcal{G}_E}(C_{E,S(E)}).
\end{align*}
Then the pro-$p$-completion of $\VS(\cK/K)$ is equal to
\begin{align*}
    \varprojlim_{n \geq 1} \left(\varprojlim_{E} \Det_{\mathcal{G}_E}(C_{E,S(E)})\right)/p^n &\xrightarrow{\sim} \varprojlim_{n \geq 1} \varprojlim_{E} (\Det_{\mathcal{G}_E}(C_{E,S(E)})/p^n)\\
    &\xrightarrow{\sim} \varprojlim_{E} \varprojlim_{n \geq 1}(\Det_{\mathcal{G}_E}(C_{E,S(E)})/p^n)\\
    &\xrightarrow{\sim}  \varprojlim_{E} \Xi_{E}
\end{align*}
where in all cases $E$ runs over $\Omega(\cK/K)$. Here the first isomorphism follows from the fact that Lemma \ref{omega-cofinal-tower-lemma} combines with the surjectivity of each map $\nu_{E'/E}$ to imply the inverse system underlying $\VS(\cK/K)$ satisfies the Mittag-Leffler property and the last from the fact that $\Det_{\mathcal{G}_E}(C_{E,S(E)})$ is finitely generated so that its pro-$p$ completion is $\Xi_{E}$. 

Lemma \ref{omega-cofinal-tower-lemma} implies that we can compute the last limit by using the tower of fields $(K_n)_{n \geq 1}$. To do this we set $\mathcal{G}_n := \mathcal{G}_{K_n}, \nu_n := \nu_{K_n/K_{n-1}}$ and $\Xi_n := \Xi_{K_n}$.

Note that each $\ZZ_p[\mathcal{G}_n]$-module $\Xi_{n}$ is free of rank one. We fix $n$ and assume that for each $m < n$ there exists a $\ZZ_p[\mathcal{G}_m]$-basis of $\Xi_{m}$ with $\nu_{K_m/K_{m-1}}(z_m) = z_{m-1}$. Then it is enough to show that $z_{n-1}$ lifts to a $\ZZ_p[\mathcal{G}_n]$-basis of $\Xi_{n}$.

To do this we write $L$ for the maximal $p$-extension inside $K_n/K_{n-1}$ and set $\mathcal{P} := \Gal(L/K_{n-1})$ and
$\mathcal{H} := \Gal(K_n/L)$.

Then if $z_L$ is any pre-image of $z_{n-1}$ under $\nu_{L/K_{n-1}}$ one has $\Xi_{L} = I(\mathcal{P})\cdot\Xi_{L} + \ZZ_p[\mathcal{G}_L]\cdot z_L$
where $I(\mathcal{P})$ is the ideal of $\ZZ_p[\mathcal{G}_L]$ generated by all elements of the form $g-1$ with $g$  in $\mathcal{P}$.

In particular, since $I(\mathcal{P})$ is contained in the Jacobson radical of $\ZZ_p[\mathcal{G}_L]$, Nakayama's Lemma implies that $z_L$ is a $\ZZ_p[\mathcal{G}_L]$-basis of $\Xi_{L}$.

Observe moreover that, since $|\mathcal{H}|$ is invertible in $\ZZ_p$, there is a canonical decomposition of $\ZZ_p[\mathcal{G}_n]$-modules
\begin{align*}
    \Xi_{n} =\, &e_{\mathcal{H}}\Xi_n \oplus (1-e_{\mathcal{H}})\Xi_n\\ =\, &(\ZZ_p[\mathcal{P}] \otimes_{\ZZ_p[\mathcal{G}_n]} \Xi_n) \oplus (1-e_{\mathcal{H}})\Xi_n\\
    \cong\, &\Xi_L \oplus (1-e_\mathcal{H})\Xi_n
\end{align*}
where $e_\mathcal{\mathcal{H}}$ denotes the idempotent $|\mathcal{H}|^{-1} \sum_{\sigma \in \mathcal{H}} \sigma$ of $\ZZ_p[\mathcal{G}_n]$.

It is thus enough to choose $z_n$ to be any basis of $\Xi_n$ that projects under this decomposition to give $z_L$ in the first component.
\end{proof}


\subsection{The construction of basic Euler systems}\label{projection-map-section} In this section we specialise to the case $\cK = K^{\rm s}$ and set $r := r_K :=|S_\infty(K)|$.

In the following, we also set $S:=S_\infty(K)$ (so that $S(E)=S_\infty(K) \cup S_{\rm ram}(E/K)$ for any $E \in \Omega(K^{\rm s}/K)$).

Then for every $E$ in $\Omega(K^{\rm s}/K)$ the $\ZZ[\cG_E]$-module $Y_E := Y_{E,S_\infty(K)}$ is free of rank $r$. In particular, by fixing a set of representatives of the $G_K$-orbits of embeddings $K^{\rm s} \to \QQ^c$ we obtain (by restriction of the embeddings) a compatible family of $\ZZ[\cG_E]$-bases of the modules $Y_E$ and hence a compatible family of isomorphisms
\begin{equation}\label{iso specified} Y_E \cong \ZZ[\cG_E]^r.\end{equation}

We also 
define an idempotent of $\QQ[\cG_E]$ by setting
\[ e_{(E)} := \sum_{\chi} e_\chi,\]
where $\chi$ runs over all characters in $\widehat{\cG_E}$ for which $X^\chi_{E,S_{\rm ram}(E/K)}$ vanishes. (Here we note that, whilst each individual idempotent $e_\chi$ belongs to $\QQ^c[\cG_E]$, the sum $e_{(E)}$ belongs to $\QQ[\cG_E]$ since $X_{E,S_{\rm ram}(E/K)}$ spans a finitely generated $\QQ[\cG_E]$-module.)

Then, with this definition, the natural exact sequence
\begin{equation}\label{can ses1} 0 \to X_{E,S_{\rm ram}(E/K)} \to X_{E,S(E)} \to Y_{E}\to 0\end{equation}
(the third arrow of which is surjective since, by assumption, $S_{\rm ram}(E/K)$ is non-empty) restricts to give an identification
\begin{equation}\label{iso specified2} e_{(E)}(\QQ\otimes_\ZZ X_{E,S(E)}) = e_{(E)}(\QQ\otimes_\ZZ Y_{E}) \cong (\QQ[\cG_E]e_{(E)})^r\end{equation}
where the isomorphism is induced by (\ref{iso specified}).

\subsubsection{Statement of the main result}\label{somr section} We define $\Theta_{E}$ to be the composite (surjective) homomorphism of $\ZZ[\mathcal{G}_E]$-modules
\begin{align}\label{rational-projection-map}
    &\QQ \otimes_\ZZ \Det_{\mathcal{G}_E}(C_{E,S(E)})\\
    \longrightarrow\, &\Det_{\QQ[\mathcal{G}_E]}(\QQ \otimes_\ZZ H^0(C_{E,S(E)})) \otimes_{\QQ[\mathcal{G}_E]} \Det^{-1}_{\QQ[\mathcal{G}_E]}(\QQ \otimes_{\ZZ} H^1(C_{E,S(E)}))\notag\\
    \xrightarrow{ e_{(E)}\times }\, &e_{(E)}\left(\QQ \otimes_\ZZ \exprod_{\mathcal{G}_E}^r \mathcal{O}_{E,S(E)}^\times\right) \otimes_{\QQ[\mathcal{G}_E]} e_{(E)}\left(\QQ \otimes_\ZZ \exprod_{\mathcal{G}_E}^r \ZZ[\mathcal{G}_E]^r\right)\notag\\
    =\, &e_{(E)}\left(\QQ \otimes_\ZZ \exprod_{\mathcal{G}_E}^r \mathcal{O}_{E,S(E)}^\times\right).\notag
\end{align}
Here the first arrow is the composite of the natural identification $\QQ \otimes_\ZZ \Det_{\mathcal{G}_E}(C_{E,S(E)}) = \Det_{\QQ[\mathcal{G}_E]}(\QQ \otimes_\ZZ C_{E,S(E)})$ and the standard `passage to cohomology' map. In addition, the second arrow is induced by multiplication by $e_{(E)}$, the isomorphism
\[ e_{(E)}(\QQ \otimes_{\ZZ} H^1(C_{E,S(E)})) \cong e_{(E)}(\QQ\otimes_\ZZ X_{E,S(E)})\cong e_{(E)}\QQ[\cG_E]^r\]
induced by
 Proposition \ref{new one}(i) and (\ref{iso specified2}) and the fact that
\[ e_{(E)}(\QQ\otimes_\ZZ H^0(C_{E,S(E)})) = e_{(E)}(\QQ\otimes_\Z \mathcal{O}_{E,S(E)}^\times)\]
is a free $\QQ[\cG_E]e_{(E)}$-module of rank $r$.

The collection of morphisms $(\Theta_E)_E$ then induces a homomorphism of $\cR^s_K$-modules
\begin{align*}
    \Theta_K^{\rm s} : \VS(K^{\rm s}/K) \to \prod_{E \in \Omega(K^{\rm s}/K)} \QQ \otimes_\ZZ \exprod_{\mathcal{G}_E}^r \mathcal{O}_{E, S(E)}^\times
\end{align*}
and we set 
\[ {\rm ES}^{\rm b}_K := \im(\Theta_K^{\rm s}).\]

Finally, we note that for every $E$ in $\Omega(K^{\rm s}/K)$ and every character $\chi$ in $\widehat{\cG_E}$ the $S(E)$-truncated Artin $L$-series $L_{S(E)}(\chi,s)$ vanishes to order at least $r$ at $s=0$ (see, for example, the discussion in \S\ref{last step proof} below) and so we can write $L^{(r)}_{S(E)}(\chi,0)$ for the value at $s=0$ of its $r$-th derivative.

We can now state the main result of this section.

\begin{theorem}\label{tech req}\
\begin{itemize}
\item[(i)] ${\rm ES}^{\rm b}_K$ is contained in ${\rm RES}_K$.
\item[(ii)] $\mathcal{A}_K^{\rm s}\cdot {\rm ES}^{\rm b}_K$ is contained in ${\rm ES}_K.$
\item[(iii)] Fix a system $c$ in ${\rm RES}_K$ and a field $E$ in $\Omega(K^{\rm s}/K)$. Then for every ramified character $\chi$ in $\widehat{\cG_E}$ one has
$$e_\chi(c_E) \not= 0 \Longrightarrow L^{(r)}_{S(E)}(\chi,0) \not= 0.$$
\item[(iv)] There exists a system $c$ in ${\rm ES}_K$ with the property that for every field $E$ in $\Omega(K^{\rm s}/K)$ and every character $\chi$ in $\widehat{\cG_E}$ one has
$$e_\chi(c_E) \not= 0 \Longleftrightarrow L^{(r)}_{S(E)}(\chi,0) \not= 0.$$
\end{itemize}
\end{theorem}


\subsubsection{The proof of Theorem \ref{tech req}(i)}
    Fix $(z_E)_E$ in ${\rm VS}(K^{\rm s}/K)$ and set $c_E := \Theta_E(z_E)$ for each $E$ in $\Omega(K^{\rm s}/K)$.

   Then to show that the family $(c_E)_E$ belongs to $\GES_K$, it suffices to prove that for every pair of fields $E$ and $E'$ in $\Omega(K^{\rm s}/K)$ with $E \subseteq E'$ and every $\chi$ in $\widehat{\cG_E}$ one has
\begin{equation}\label{needed eq}   e_\chi(\N^r_{E'/E}(c_{E'})) = e_\chi (P_{E'/E}) \cdot e_\chi(c_E) \end{equation}
with $P_{E'/E} := \prod_{v \in S(E')\backslash S(E)} (1-\Fr_v^{-1})$.

In addition, it is enough to verify this equality in $\QQ_p\otimes_\ZZ \exprod_{\mathcal{G}_E}^r \mathcal{O}_{E,S(E')}^\times$ and with the system $(c_E)_E$ replaced by the image $(\tilde c_E)_E$ under the collection of maps $(\QQ_p\otimes_{\QQ}\Theta_E)_E$ of any $\widehat{\mathcal{R}_K^{\rm s}}^p$-generator of the pro-$p$-completion of ${\rm VS}(K^{\rm s}/K)$ (cf. Proposition \ref{invert prop}).

Then, for every $E$ and every $\chi$ in $\widehat{\cG_E}$ the surjectivity of $\Theta_E$ implies that
\[ e_\chi(\tilde c_E) \not= 0 \Longleftrightarrow e_\chi e_{(E)} \not= 0 \Longleftrightarrow X^\chi_{E,S_{\rm ram}(E/K)} = 0.\]
Thus, the direct sum decomposition $ X_{E,S_{\rm ram}(E'/K)} = X_{E,S_{\rm ram}(E/K)} \oplus Y_{E,S(E')\setminus S(E)}$ implies
\begin{align*} e_\chi(\N^r_{E'/E}(\tilde c_{E'})) \not= 0 \Longleftrightarrow &\,\, X^\chi_{E,S_{\rm ram}(E'/K)} = 0\\
\Longleftrightarrow &\,\, X^\chi_{E,S_{\rm ram}(E/K)} = 0 \,\,\text{ and }\,\, Y^\chi_{E,S(E')\setminus S(E)} = 0\\
\Longleftrightarrow &\,\, e_\chi(\tilde c_E) \not= 0 \,\,\text{ and }\,\, e_\chi(P_{E'/E}) \not= 0.\end{align*}

It therefore suffices to verify (\ref{needed eq}) for characters $\chi$ with $e_\chi(P_{E'/E}) \neq 0$ and in this case the required equality  follows directly from the following commutative diagram

\[ \begin{CD}  \left(\Det_{\mathcal{G}_{E'}}(C_{E',S(E')})\right)^{\chi} @> \Theta_{E'}>> \displaystyle\left(\exprod_{\mathcal{G}_{E'}}^r \mathcal{O}_{E',S(E')}^\times\right)^{\chi}\\
@V \nu_{E'/E} VV @VV \N^r_{E'/E}V \\
  \left(\Det_{\mathcal{G}_E}(C_{E,S(E)})\right)^{\chi} @> P_{E'/E}\cdot\Theta_{E}>> \displaystyle\left(\exprod_{\mathcal{G}_E}^r \mathcal{O}_{E,S(E)}^\times\right)^{\chi}.\end{CD}\]

The existence of this diagram follows from the fact (itself a consequence of the general observation in \cite[Lem. 1]{bf}) that for each $v \in S(E')\backslash S(E)$ there is in this case a commutative diagram
\[ \begin{CD}
 \Det_{\mathcal{G}_E}(\DR\Gamma(\kappa(v)_\mathcal{W}, \ZZ[\cG_E])^\ast[-1])^{\chi} @> >> \ZZ[\cG_E]^\chi \\
 @\vert @VV 1-\chi({\rm Fr}_v) V\\
 \Det_{\mathcal{G}_E}(\DR\Gamma(\kappa(v)_\mathcal{W}, \ZZ[\cG_E])^\ast[-1])^{\chi} @> >> \ZZ[\cG_E]^\chi
\end{CD}
\]
where the upper row is induced by the isomorphism (\ref{triv triv}) and the lower row by the acyclicity  of the complex $(\DR\Gamma(\kappa(v)_\mathcal{W}, \ZZ[\cG_E])^\ast[-1])^{\chi}$.


\subsubsection{The proof of Theorem \ref{tech req}(ii)}\label{proof tech2} For $a = (a_E)_E$ in $\cA_K^s$ and $z = (z_E)_E$ in $\VS(K^{\rm s}/K)$ we need to show that for every fixed $E$ in $\Omega(K^{\rm s}/K)$ one has
\begin{equation}\label{needed again} a_E\cdot\Theta_E(z_E) \in \bidual_{\cG_E}^r \cO_{E,S(E)}^\times . \end{equation}

In view of Lemma \ref{torsion lemma} below, we can also assume $a_E=\delta_T := 1-\mathrm{N}v\cdot\Fr_v^{-1}$, with $T = \{v\}$, where $v$ is a place of $K$ that is not contained in $S(E)$ and $\mathcal{O}_{E,S(E),T}^\times$ is torsion free.

To prove this we note that
\[  \Det_{\cG_E}(C_{E,S(E),T}) = \Det_{\cG_E}^{-1}(\mathbb{F}_{T_E}^\times[0])\cdot\Det_{\cG_E}(C_{E,S(E)})  = \delta_T\cdot\Det_{\cG_E}(C_{E,S(E)}) \]
where the first equality follows from the exact triangle in Proposition \ref{new one}(ii) (with $T$ and $T'$ replaced by $\emptyset$ and $T$ respectively) and the second from the fact the explicit resolution of  $\mathbb{F}_{T_E}^\times[0]$ described in Remark \ref{F-T-resolution-remark} implies that $\Det_{\cG_E}(\mathbb{F}_{T_E}^\times[0]) = \ZZ[\cG_E]\cdot\delta_T^{-1}$.

Hence, since $\bidual_{\cG_E}^r \cO_{E,S(E),T}^\times$ is a subset of $\bidual_{\cG_E}^r \cO_{E,S(E)}^\times$ it is enough to show that for every prime $p$ one has
\[ (\QQ_p\otimes_\QQ\Theta_E)(\Det_{\ZZ_p[\cG_E]}(\ZZ_p\otimes_\ZZ C_{E,S(E),T})) \subseteq \bidual_{\ZZ_p[\cG_E]}^r (\ZZ_p\otimes_\ZZ\cO_{E,S(E),T}^\times).\]

To do this we note that, since $\cO_{E,S(E),T}^\times$ is torsion-free, Proposition \ref{new one}(i) implies that $\ZZ_p\otimes_\ZZ C_{E,S(E),T}$ is an admissible complex of $\ZZ_p[\cG_E]$-modules in the sense of \cite[Def. 2.20]{sbA}.

In particular, the above inclusion follows directly upon applying \cite[Prop. A.11(ii)]{sbA} with the data $(\mathcal{R},C,X)$ taken to be
$(\ZZ_p[\cG_E],\ZZ_p\otimes_\ZZ C_{E,S(E),T},\ZZ_p\otimes_\ZZ Y_E)$ and the map $f$ equal to the natural composite homomorphism
\[ H^1(\ZZ_p\otimes_\ZZ C_{E,S(E),T}) \to \ZZ_p\otimes H^1(C_{E,S(E),T})_{\rm tf} \cong \ZZ_p\otimes_\ZZ X_{E,S(E)} \to \ZZ_p\otimes_\ZZ Y_E.\]
where the isomorphism is by Proposition \ref{new one}(i) and the last map is the surjective map induced by (\ref{can ses1}).

This completes the proof of Theorem \ref{tech req}(ii).

\begin{lemma}\label{torsion lemma} For each field $E$ in $\Omega(K^{\rm s}/K)$ the following claims are valid.

\begin{itemize}
\item[(i)] Let $U$ be any finite set of places of $K$ containing $S(E)$ and all places dividing $|\mu_E|$.
Then $\mathcal{A}_E$ is generated as a $\ZZ$-module by $\{1-\mathrm{N}v\cdot\Fr_v^{-1} \mid v \notin U\}$. Furthermore, for any place $v$ of $K$ that does not belong to $U$, the group $\mathcal{O}_{E,S(E),\{v\}}^\times$ is torsion-free.
\item[(ii)] The natural projection map $\mathcal{A}_K^{\rm s} \to \mathcal{A}_E$ is surjective.
\end{itemize}
\end{lemma}

\begin{proof} The first assertion of claim (i) is proved in \cite[Chap. IV, Lem. 1.1]{tate} and the second assertion is a straightforward exercise since the residue characteristic of $v$  does not divide 
$|\mu_E|$.

To prove claim (ii) it suffices to show that for any fields $E$ and $E'$ in $\Omega(K^{\rm s}/K)$ with $E \subseteq E'$ the natural projection $\ZZ[\cG_{E'}] \to \ZZ[\cG_E]$ sends $\mathcal{A}_{E'}$ onto $\mathcal{A}_E$.

This follows easily by applying the first assertion of claim (i) for both $E'$ and $E$ with respect to the same set $U$ in both cases. \end{proof}

\subsubsection{The proof of Theorem \ref{tech req}(iii) and (iv)}\label{last step proof}

We fix $E$ in $\Omega(K^{\rm s}/K)$ and a character $\chi$ in $\widehat{\cG_E}$ that is ramified.

In this case the fixed field $E_\chi$ of $E$ by the kernel of $\chi$ belongs to $\Omega(K^{\rm s}/K)$ and so for any system $c = (c_E)_E$ in ${\rm RES}_K$ one has
\begin{align*} [E:E_\chi]\cdot e_\chi (c_E) =\, &e_\chi(\N^r_{E/E_\chi}(c_E))\\
 =\, &e_\chi(P_{E/E_\chi})\cdot e_\chi(c_{E_\chi})\\
  =\, &\left(\prod_{v \in S(E)\setminus S(E_\chi)}(1-\chi(\Fr_v^{-1}))\right)\cdot e_\chi(c_{E_\chi}).\end{align*}
In particular, if $e_\chi (c_E)\not= 0 $, then one has $\chi(\Fr_v^{-1})\not= 1$ for all $v$ in $S(E)\setminus S(E_\chi)$.

On the other hand, since $L_{S(E_\chi)}(\chi, s) = L(\chi, s)$ one has
\begin{equation}\label{changeS}
        L_{S(E)}(\chi, s) = \left(\prod_{v \in S(E)\backslash S(E_\chi)} (1-\chi(\Fr_v)\mathrm{N}v^{-s} )\right)\cdot L(\chi,s),\end{equation}
whilst \cite[Chap. I, Prop. 3.4]{tate}
implies that if $\Sigma$ denotes either $S(E)$ or $S_\infty(K)$, then for all $\psi$ in $\widehat{\cG_E}$ one has 
\begin{align}\label{ord formula}  {\rm ord}_{s=0}L_{\Sigma}(\psi, s) =\, &{\rm dim}_{\QQ^c}(X_{E,\Sigma}^\psi)\\
  =\, &r +\begin{cases}  {\rm dim}_{\QQ^c}(X_{E,S_{\rm ram}(E/K)}^\psi), &\text{if $\Sigma = S(E)$,}\\
0, &\text{if $\Sigma=S_\infty(K)$ and $\psi$ is non-trivial.}\end{cases}\notag\end{align}
(Note that if $\Sigma = S(E)$, then the second equality here is valid for the trivial character $\psi$ since, by assumption, $S_{\rm ram}(E/K)$ is not empty.)

In particular, since the ramified character $\chi$ cannot be trivial, this implies $L^{(r)}(\chi, 0)\not= 0$ and hence that the value 
\[ L^{(r)}_{S(E)}(\chi, 0) = \left(\prod_{v \in S(E)\setminus S(E_\chi)}(1-\chi(\Fr_v^{-1}))\right)\cdot L^{(r)}(\chi, 0)\]
is not zero if and only if $\chi(\Fr_v^{-1})\not= 1$ for all $v$ in $S(E)\setminus S(E_\chi)$. Since we have observed that the latter condition is satisfied whenever $e_\chi (c_E)\not= 0 $, this proves Theorem \ref{tech req}(iii).

We claim next that to prove Theorem \ref{tech req}(iv) it is enough to show that there exist elements $a = (a_E)_E$ of $\cA_K^{\rm s}$ and $z = (z_E)_E$ of $\VS(K^{\rm s}/K)$ with the property that for every $E$ in $\Omega(K^{\rm s}/K)$ one has $a_E\in \QQ[\cG_E]^\times$ and $\QQ[\cG_E]\cdot z_E = \QQ\otimes_\ZZ\Det_{\cG_E}(C_{E,S(E)})$.

In fact, if this is true then Theorem \ref{tech req}(ii) combines with the argument in the proof of Theorem \ref{tech req}(i) and (\ref{ord formula}) to show that $c := a\cdot \Theta_K^s(z)$ is a system in ${\rm ES}_K$ with the property that
 for all $E$ in $\Omega(K^{\rm s}/K)$ and all $\chi$ in $\widehat{\cG_E}$ one has $e_\chi(c_E)\not= 0$ if and only if
 $L^{(r)}_{S(E)}(\chi, 0)\not= 0$.

To complete the proof of Theorem \ref{tech req} it is thus enough to construct elements $a$ in $\cA_K^{\rm s}$ and $z$ in $\VS(K^{\rm s}/K)$
         with the properties described above.

To do this we again use Lemma \ref{omega-cofinal-tower-lemma} to reduce to consideration of the tower of fields $(K_n)_{n \geq 1}$. Then, to ease notation, we set $\cG_{n} := \cG_{K_n}$ $\cH_n :=\Gal(K_n/K_{n-1})$, $\Xi_n := \Det_{\cG_n}(C_{K_n, S(K_n)})$ and $\nu_n := \nu_{K_n/K_{n-1}}$ for each $n$ in $\NN$.

To construct $z$ we fix $n$ in $\mathbb{N}$ and assume that for each $m < n$ we have fixed an element $z_m$ of $\Xi_m$ such that $\QQ\otimes_\ZZ\Xi_m = \QQ[\cG_m]\cdot z_m$ and $\nu_{m}(z_m) = z_{m-1}$. Then we must construct an element $z_n$ of $\Xi_n$ such that $\QQ\otimes_\ZZ\Xi_n = \QQ[\cG_n]\cdot z_n$ and $\nu_{n}(z_n) = z_{n-1}$

To do this we first choose any pre-image $z'_n$ of $z_{n-1}$ under the (surjective) map $\nu_{K_n/K_{n-1}}$. Then, since $\nu_{K_n/K_{n-1}}$ induces an isomorphism
\begin{align*}  \QQ\otimes_\ZZ\Xi_n =\, &e_{\cH_n}\cdot (\QQ\otimes_\ZZ\Xi_n) \oplus (1-e_{\cH_n})(\QQ\otimes_\ZZ\Xi_n)\\
 \cong \, &(\QQ\otimes_\ZZ\Xi_{n-1}) \oplus (1-e_{\cH_n})(\QQ\otimes_\ZZ\Xi_n)\end{align*}
there exists an element $z_n''$ of $\Xi_n$ such that $e_{\cH_n}(z_n'') = 0$ and $z_n'+ z''_n$ is a $\QQ[\cG_n]$-generator of $\QQ\otimes_\ZZ\Xi_n$.
 The element $z_n := z_n'+ z''_n$ is then an element of the required type.

By an entirely similar argument (which we leave to the reader), one finds that Lemma \ref{torsion lemma}(ii) implies the existence of an element $a$ in $\mathcal{A}^{\rm s}_K$ with the required property.

This completes the proof of Theorem \ref{tech req}.


\section{Higher Kolyvagin derivatives}\label{koly section}

In this section we use the theory of equivariant higher Kolyvagin derivatives to obtain some strong, and unconditional, evidence in support of Conjecture \ref{main conj}.

\subsection{Statement of the main result}\label{state precise}

\subsubsection{}\label{H hyp sec}To state a precise version of Theorem \ref{koly pre thm} we fix an {\it odd} prime $p$ and a finite abelian $p$-extension $F$ of $K$. For simplicity, we shall assume that $F$ contains the maximal $p$-extension $K(1)$ of $K$ inside its Hilbert class field $H_K$.

As before, we continue to set $S:=S_\infty(K)$ and we recall this means that for each $E'$ in $\Omega(K^{\rm s}/K)$ the notation $S(E')$ denotes $S_\infty(K)\cup S_{\rm ram}(E'/K)$.

We write $G_K$ for the absolute Galois group of $K$ and $\omega$ for the $p$-adic Teichm\"uller character.

We also fix a non-trivial character
\[ \chi: G_K \to \QQ_p^{c,\times}\]
of finite prime-to-$p$ order, write $L$ for the abelian extension of $K$ that corresponds to the kernel of $\chi$ and assume that all of the following conditions are satisfied:
\begin{itemize}
\item[(H$_1$)] $L$ is not contained in $K(\mu_p)$.
\item[(H$_2$)] $\chi^2 \neq \omega$ if $p=3$.
\item[(H$_3$)] any place of $K$ that ramifies in $F$ is not completely split in $L$.
\item[(H$_4$)] all archimedean places of $K$ split in $L$.
\item[(H$_5$)] $L/K$ is ramified.
\end{itemize}

In the sequel we shall consider the compositum
\[ E := LF\]
of $L$ and $F$. We note, in particular that, since $p$ is odd, hypothesis (H$_4$) implies that $E$ is contained in $K^{\rm s}$.

We decompose the group $\cG_E$ as a product $\Pi\times \Delta$ with $\Pi$ the Sylow $p$-subgroup of $\cG_E$ and we identify $\cG_F$ and $\cG_L$ with $\Pi$ and $\Delta$ in the obvious way.


We set $\cO:=\ZZ_p[\im \chi]$ and define the $(p,\chi)$-component of a $\Delta$-module $X$ by setting
\[ X_\chi:=\cO \otimes_{\ZZ[\Delta]} X,\]
where $\cO$ is regarded as a $\ZZ[\Delta]$-algebra via $\chi$. 
 For an element $a \in X$, we also set
\[ a^\chi:=1 \otimes a \in X_\chi.\]

The following result is a precise version of Theorem \ref{koly pre thm}.

\begin{theorem}\label{koly thm} For any system $c$ in ${\rm ES}_K$, any field $F$ as above and any character $\chi$ that satisfies all of the hypotheses (H$_1$), (H$_2$), (H$_3$), (H$_4$) and (H$_5$) one has $c^\chi_E \in \Theta_E(\Det_{\cG_E}(C_{E,S(E)}))_\chi$. 
\end{theorem}

In the remainder of \S\ref{state precise} we shall reduce the proof of this result to the proof of a statement about higher rank Stark systems for $F/K$ and $\chi$.

\subsubsection{}We fix an abelian pro-$p$ extension $\cK$ of $K$ that contains both $F$ and the maximal $p$-extension $K(\fq)$ of $K$ inside its ray class field modulo $\fq$ for all but finitely many primes $\fq $ of $K$. We note that, since $p$ is odd, $\cK$ is a subfield of $K^{\rm s}$.

We write $\Omega'(\cK/K)$ for the set of {\em all} finite extensions of $K$ in $\cK$ and for each $F'$ in $\Omega'(\cK/K)$ we set $U_{F'}:=(\cO_{LF'}^\times)_\chi$.

We continue to write $r$ in place of $r_K:=|S_\infty(K)|$. We recall from \cite[Def. 2.4]{bss3} that a `strict $p$-adic Euler system' of rank $r$ for the extension $\cK/K$ and character $\chi$ is a collection
\[c=(c_{F'})_{F'} \in \prod_{F' \in \Omega'(\cK/K)} {\bigcap}_{\cO[\cG_{F'}]}^r U_{F'}\]
that for all $F'$ and $F''$ in $\Omega'(\cK/K)$ with $F'\subseteq F''$ satisfies the distribution relation (\ref{dist rel}) with $E'/E$ replaced by $F''/F'$. The set of strict $p$-adic Euler systems of rank $r$ for $\cK/K$ and $\chi$ is denoted ${\rm ES}_r(\cK/K,\chi)$  and is naturally an $\cO[[\Gal(\cK/K)]]$-module.

We further recall from \cite[\S2.3]{bss3} that a Kolyvagin system of rank $r$ for $F/K$ and $\chi$ is a collection of elements that is parametrised by certain square-free products of prime ideals of $K$ and that the set ${\rm KS}_r(F/K,\chi)$ of all such collections is naturally a module over the quotient $\mathcal{O}[\Pi]$ of $\cO[[\Gal(\cK/K)]]$.

In addition, the main result of \cite{bss3} implies that, under the present hypotheses on $\cK/K$ and $\chi$, the $\mathcal{O}[\Pi]$-module ${\rm KS}_r(F/K,\chi)$ is free of rank one and there exists a canonical `$F$-relative $r$-th order Kolyvagin derivative' homomorphism of $\cO[[\Gal(\cK/K)]]$-modules
\[ \mathcal{D}_{F,r}: {\rm ES}_r(\cK/K,\chi) \to {\rm KS}_r(F/K,\chi).\]

 We write ${\rm VS}_{L\cK}^{(p)}$ for the pro-$p$-completion of ${\rm VS}(L\cK/K)$.

\begin{lemma}\label{first step to proof} \
\begin{itemize}
\item[(i)] For $c$ in ${\rm ES}_K$ the assignment $F'\mapsto c_{LF'}^\chi$ defines an element $c^\chi_{\rm str}$ of ${\rm ES}_r(\cK/K,\chi)$.
\item[(ii)] There exists a natural homomorphism of $\widehat{\cR}^p_{K^{\rm s}/K}$-modules
\[ \Theta_{L\cK}^\chi: {\rm VS}_{L\cK}^{(p)} \to {\rm ES}_r(\cK/K,\chi).\]
\item[(iii)] Let $b$ be a generator of ${\rm VS}_{L\cK}^{(p)}$ over $\widehat{\cR}^p_{K^{\rm s}/K}$. Then Theorem \ref{koly thm} is valid if the $\cO[\Pi]$-module ${\rm KS}_r(F/K,\chi)$ is generated by the image of $\Theta^\chi_{L\cK}(b)$ under $\mathcal{D}_{F,r}$.
\end{itemize}
\end{lemma}

\begin{proof} Hypotheses (H$_5$) implies that for each $F'$ in $\Omega'(\cK/K)$ the field $LF'$ belongs to $\Omega(L\cK /K)$. Thus, if $c$ belongs to ${\rm ES}_K$, then for $F'$ in $\Omega'(\cK/K)$ the element $c_{{\rm str},F'}^\chi := c_{LF'}^\chi$ belongs to ${\bigcap}_{\cO[\cG_{F'}]}^r(\mathcal{O}^\times_{LF',S(LF')})_\chi$ and the collection
\[ c_{{\rm str}}^\chi := ( c_{{\rm str},F'}^\chi)_{F'}\]
satisfies the necessary distribution relations as the fields $F'$ vary.

 To deduce $c_{{\rm str}}^\chi$ belongs to ${\rm ES}_r(\cK/K,\chi)$ it is thus enough to show that for each $F'$
   the element $c_{LF'}^\chi$ belongs to ${\bigcap}_{\cO[\cG_{F'}]}^rU_{F'} $.

In view of the equality (\ref{ord formula}), it is therefore sufficient to show that $e_{\psi \chi} (c_{LF'})=0$ for any character $\psi$ in $ \widehat{\cG_{F'}}$ for which $L^{(r)}_{S(LF')}(\psi\chi,0)=0$, where we regard each product $\psi \chi$ as a character in $\widehat {\cG_{LF'}}$ in the obvious way.

In addition, since (H$_5$) implies $\psi\chi$ is ramified, the required vanishing is a direct consequence of Theorem \ref{tech req}(iii). This proves claim (i).


To prove claim (ii) we show first that for each $F'$ in $\Omega'(\cK/K)$ the map $\Theta_{LF'}^\dag:= \QQ_p^c\otimes_\QQ\Theta_{LF'}$ sends $\Xi_{F',\chi} := \Det_{\cG_{LF'}}(C_{LF',S(LF')})_\chi$ to ${\bigcap}_{\cO[\cG_{F'}]}^rU_{F'}$.

To show this we note that (H$_1$) implies $\mu_{LF',\chi}$ vanishes so that $\mathcal{A}_{LF',\chi}$ is equal to $\ZZ[\cG_{LF'}]_\chi$ and hence contains the element $e_\chi$.

From the argument of Theorem \ref{tech req}(iv) we can therefore deduce $\Theta_{LF'}^\dag(\Xi_{F',\chi})$ is contained in ${\bigcap}_{\cO[\cG_{F'}]}^r U_{F'}$. 

We can therefore define $\Theta_{L\cK}^\chi$ to be the product map $\prod_{F'}\Theta^\chi_{F'}$ where $F'$ runs over $\Omega'(\cK/K)$ and each $\Theta^\chi_{F'}$ is the composite homomorphism
\[ {\rm VS}_{L\cK}^{(p)} \to \Det_{\cG_{LF'}}(C_{LF',S(LF')})_\chi \xrightarrow{\Theta^\dag_{LF'}} {\bigcap}_{\cO[\cG_{F'}]}^rU_{F'}\]
where the first map is the natural projection.

To prove claim (iii) we assume $\mathcal{D}_{F,r}(\Theta^\chi_{L\cK}(b))$ is a generator over $\cO[\Pi]$ of ${\rm KS}_r(F/K,\chi)$.
 Then for any $c$ in ${\rm ES}_K$ claim (i) implies there exists an element $x = x_{b,c}$ of $\mathcal{O}[\Pi]$ such that $\mathcal{D}_{F,r}(c^\chi_{\rm str}) = x\cdot \mathcal{D}_{F,r}(\Theta^\chi_{L\cK}(b))$.

 Upon evaluating these (equivariant) Kolyvagin systems at $1$ (that is, at the empty product of prime ideals of $K$) one deduces that %
\[ c_E^\chi  = c^\chi_{{\rm str},F} = \mathcal{D}_{F,r}(c^\chi_{\rm str})_{1} = x\cdot \mathcal{D}_{F,r}(\Theta^\chi_{L\cK}(b))_{1} = x\cdot  \Theta_{L\cK}^\chi(b)_{F} = x\cdot \Theta_{E}(b_E)^\chi\]
where $b_E$ denotes the image of $b$ in $\ZZ_p\otimes_\ZZ\Det_{\cG_{E}}(C_{E,S(E)})$. This shows that $c^\chi_E $ belongs to
 $\Theta_E(\Det_{\cG_E}(C_{E,S(E)}))_\chi$, as required.
\end{proof}

In view of Lemma \ref{first step to proof}(ii), to prove Theorem \ref{koly thm} it is enough to show that $\mathcal{D}_{F,r}(\Theta^\chi_{L\cK}(b))$ generates ${\rm KS}_r(F/K,\chi)$ over $\cO[\Pi]$.


To do this we use the $\mathcal{O}[\Pi]$-module ${\rm SS}_r(F/K,\chi)$ of Stark systems of rank $r$ for $F/K$ and $\chi$ that is defined in \cite[\S2.3]{bss3} and recall from \cite[Th. 3.3(ii)]{bss3} that there exists a canonical `algebraic regulator' isomorphism of $\mathcal{O}[\Pi]$-modules of the form
\[ \mathcal{R}^\chi_r: {\rm SS}_r(F/K,\chi) \xrightarrow{\sim} {\rm KS}_r(F/K,\chi).\]

The key now is to prove the following result (which, we note, also justifies \cite[Rem. 3.5]{bss3}). We abbreviate $\Gal(L\cK/K)$ to $\cG_{L\cK}$.

\begin{theorem}\label{key diagram} There exists a canonical surjective homomorphism
\[ \Delta^\chi_{F,r}: {\rm VS}_{L\cK}^{(p)} \to {\rm SS}_r(F/K,\chi)\]
of $\cO[[\cG_{L\cK}]]$-modules that makes the following diagram commute
\[\begin{CD}
{\rm VS}_{L\cK}^{(p)} @> \Theta_{L\cK}^\chi >> {\rm ES}_r(\cK/K,\chi)\\
@V \Delta_{F,r}^\chi VV @VV \mathcal{D}_{F,r} V\\
{\rm SS}_r(F/K,\chi) @> \mathcal{R}_r^\chi >> {\rm KS}_r(F/K,\chi).\end{CD}\]
\end{theorem}

Before proving this result, we note that it does indeed finish the proof of Theorem \ref{koly thm} since the surjectivity of the composite map $\mathcal{R}_r^\chi\circ\Delta_{F,r}^\chi = \mathcal{D}_{F,r}\circ\Theta_{L\cK}^\chi$ implies that the image $\mathcal{D}_{F,r}(\Theta^\chi_{L\cK}(b))$ of $b$ under this map must be a
generator of the $\cO[\Pi]$-module ${\rm KS}_r(F/K,\chi)$, as required.

\subsection{Stark systems and the proof of Theorem \ref{key diagram}} 

We first quickly review the definition of Stark systems given in \cite{bss3}.

\subsubsection{}Let $m$ be a non-negative integer and write $\cP_m$ for the set of prime ideals of $K$ that do not divide $p$ and split completely in $H_KE(\mu_{p^m}, (\cO_K^\times)^{p^{-m}})$, where $(\cO_K^\times)^{p^{-m}}$ denotes the set of elements $x$ in $\QQ^c$ with $x^{p^m} \in \mathcal{O}_K^\times$. (Recall here that $H_K$ denotes the Hilbert class field of $K$ and $E$ the field $LF$.)

We write $\cN_m$ for the set of square-free products of primes in $\cP_m.$
%
%
 For each product $\fn$ in $\cN_m$, we write $\Sigma_\fn$ for the union of $S_\infty(K)$ and the set of prime divisors of $\fn$, we set
\[ U_{E,\fn} := \mathcal{O}_{E,\Sigma_\fn}^\times\]
and we write $\nu(\fn)$ for the number of prime divisors of $\fn$.

For any prime $\fq$ of $K$, we define
\begin{eqnarray}\label{def v}
v_\fq: E^\times \to \ZZ[\cG_E]; \ a \mapsto \sum_{\sigma \in \cG_E} {\rm ord}_\fQ(\sigma a )\sigma^{-1},
\end{eqnarray}
where $\fQ$ is a fixed place of $E$ lying above $\fq$ and ${\rm ord}_\fQ$ denotes the normalized additive valuation at $\fQ$.

We set $\mathcal{O}_m:= \mathcal{O}/(p^m)$. Upon reduction modulo $p^m$, the above map induces a map of $\cO_m[\Pi]$-modules $(E^\times/p^m)_\chi \to \cO_m[\Pi]$ that we also denote by $v_\fq$ and we then set
$$\cS_m^\fn:=\{a \in (E^\times/p^m)_\chi \mid v_\fq(a)=0 \text{ for every }\fq \nmid \fn\}.$$

For each pair of products $\fm$ and $\fn$ in $\cN_m$ with $\fn \mid \fm$, there is then an exact sequence of $\cO_m[\Pi]$-modules
$$0 \to \cS_m^\fn \to \cS_m^\fm \xrightarrow{\bigoplus_{\fq \mid \fm/\fn}v_\fq}\bigoplus_{\fq \mid \fm/\fn} \cO_m[\Pi]$$
and hence (via the general result of \cite[Prop.~A.3]{sbA}), for any non-negative integer $r$ a map of $\cO_m[\Pi]$-modules
$$v_{\fm,\fn}:=\pm {\bigwedge}_{\fq \mid \fm/\fn} v_\fq: {\bigcap}_{\cO_m[\Pi]}^{r+\nu(\fm)}\cS_m^\fm \to {\bigcap}_{\cO_m[\Pi]}^{r+\nu(\fn)}\cS_m^\fn .$$
Here the sign is chosen so that $v_{\fm',\fn}=v_{\fm,\fn}\circ v_{\fm',\fm}$ if $\fn \mid \fm \mid \fm'$ (cf. \cite[\S 3.1]{sbA}). 

Then the $\cO_m[\Pi]$-module of Stark systems of rank $r$ and level $m$ for $F/K$ and $\chi$ is defined to be the inverse limit
$$ {\rm SS}_r(F/K,\chi)_m:=\varprojlim_{\fn \in \cN_m} {\bigcap}_{\cO_m[\Pi]}^{r+\nu(\fn)}\cS_m^\fn$$
with transition maps $v_{\fm,\fn}$.

We also recall, from \cite[\S2.3]{bss3}, that the present hypotheses on $\chi$ imply that for each non-negative integer $m$ there exists a natural surjective homomorphism of $\mathcal{O}_{m+1}[\Pi]$-modules
$\pi_m: {\rm SS}_r(F/K,\chi)_{m+1} \twoheadrightarrow {\rm SS}_r(F/K,\chi)_m$ so that one can set
$${\rm SS}_r(F/K,\chi):=\varprojlim_m {\rm SS}_r(F/K,\chi)_m ,$$
where the limit is taken with respect to the transition morphisms $\pi_m$.

We next note the argument of \cite[Prop. 3.6]{sano} shows that, just as above, the maps $v_{\fq}$ in (\ref{def v}) give rise to a homomorphism of $\mathcal{O}[\Pi]$-modules
\[ v_{\fm,\fn}: {\bigcap}_{\cO[\Pi]}^{r+\nu(\fm)} U_{E,\fm,\chi} \to {\bigcap}_{\cO[\Pi]}^{r+\nu(\fn)}U_{E,\fn,\chi} \]
and we define the $\cO[\Pi]$-module of global Stark systems of rank $r$ for $F/K$ and $\chi$ to be

\[ {\rm SS}^{\rm glob}_r(F/K,\chi):= \varprojlim_{\fn \in \cN_0} {\bigcap}_{\cO[\Pi]}^{r+\nu(\fn)} U_{E,\fn,\chi},\]
where the limit is taken with respect to the maps $v_{\fm,\fn}$.

\begin{lemma}\label{partial H map} There exists a natural surjective homomorphism of $\mathcal{O}[\Pi]$-modules
 \[ \varrho_{F,\chi}:{\rm SS}^{\rm glob}_r(F/K,\chi) \to {\rm SS}_r(F/K,\chi).\]
\end{lemma}

\begin{proof} For each natural number $m$ we use the cofinal subset $\cN_m'$ of $\cN_m$ described in Lemma \ref{cofinal} below.

Then the latter result implies that for each $\fm$ in $\cN_m'$ there exists a natural surjective homomorphism of $\mathcal{O}[\Pi]$-modules
\[ {\bigcap}_{\cO[\Pi]}^{r+\nu(\fm)}U_{E,\fm,\chi}\! =\! {\bigwedge}_{\cO[\Pi]}^{r+\nu(\fm)} U_{E,\fm,\chi}\!\twoheadrightarrow \!{\bigwedge}_{\cO_m[\Pi]}^{r+\nu(\fm)}(U_{E,\fm,\chi}/(p^m))\! = \! {\bigwedge}_{\cO_m[\Pi]}^{r+\nu(\fm)}\cS_m^\fm\! = \! {\bigcap}_{\cO_m[\Pi]}^{r+\nu(\fm)}\cS_m^\fm \]
that are compatible with the respective transition morphisms $v_{\fm,\fn}$ as $\fn$ and $\fm$ vary over $\cN_m'$. This therefore gives a surjective homomorphism $\varrho_{F,\chi,m}$ of $\mathcal{O}[\Pi]$-modules from
 ${\rm SS}^{\rm glob}_r(F/K,\chi)$ to ${\rm SS}_r(F/K,\chi)_m$.

These homomorphisms $\varrho_{F,\chi,m}$ are compatible with the transition morphisms $\pi_m$ as $m$ varies and hence lead to a surjective homomorphism $\varrho_{F,\chi}$ of the required sort.\end{proof}

\begin{lemma}\label{cofinal} For each natural number $m$ there exists a cofinal subset $\cN'_m$ of $\cN_m$ such that for every product $\fm$ in $\cN'_m$ the following conditions are satisfied.
\begin{itemize}
\item[(i)] The $\cO[\Pi]$-module $U_{E,\fm,\chi}$ is free of rank $r + \nu(\fm)$.
\item[(ii)] The $\co_m[\Pi]$-module $\cS_m^\fm$ is equal to $(U_{E,\fm}/p^m)_\chi$ and is free of rank $r + \nu(\fm)$.
\end{itemize}
\end{lemma}

\begin{proof} For each number field $N$ we set $N_m := N(\mu_{p^m})$. We also write $H_E^\chi$ for the subfield of the Hilbert class field of $E$ that corresponds to the image of ${\rm Cl}(\mathcal{O}_E)$ in ${\rm Cl}(\mathcal{O}_E)_\chi$.

We claim that $H_E^\chi\cap H_KE_m((\cO_K^\times)^{p^{-m}}) = E$ and to prove this we shall use the following diagram of fields.

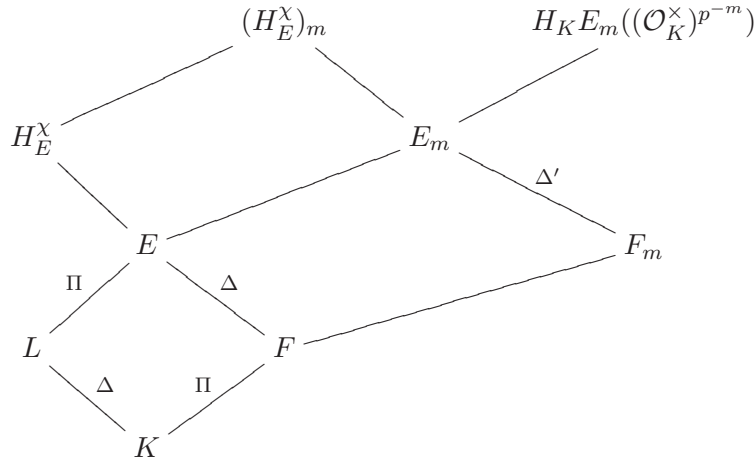
\[
\xymatrix{
& & (H_{E}^\chi)_m & & H_KE_m((\cO_K^\times)^{p^{-m}})\\
H_E^\chi\ar@{-}[dr] \ar@{-}[rru] & & & E_m\ar@{-}[dr]^{\Delta'} \ar@{-}[ul] \ar@{-}[ru]\\
& E\ar@{-}[dr]^{\Delta}\ar@{-}[rru] & & & F_m\\
L \ar@{-}[ur]^{\Pi} \ar@{-}[dr]^{\Delta} & & F \ar@{-}[urr]\\
& K \ar@{-}[ur]^{\Pi}}\]

We set $D := (H_E^\chi)_m\cap H_KE_m((\cO_K^\times)^{p^{-m}})$ and at the outset consider the conjugation action of $\Delta'$ on the abelian group $\Gal(D/E_m)$.

In fact, since $D$ is contained in $H_KE_m((\cO_K^\times)^{p^{-m}})$ it is abelian over $F_m$ and so this conjugation action must be trivial.

However, since $H_E^\chi/E$ is of $p$-power degree, $\Gal(D/E_m)$ is a $p$-group and so identifies with a subset of the module $\mathcal{O}\otimes_\ZZ \Gal(D/E_m)$ upon which the conjugation action of $\Delta'$ is via the restriction $\chi'$ of $\chi$ (as $D \subseteq (H_{E}^\chi)_m$).

In particular, since $\Delta'$, and hence also $\chi'$, is non-trivial (as $L\nsubseteq K(\mu_p)$ by hypothesis (H$_1$)), the group $\Gal(D/E_m)$ must be trivial and so $(H_E^\chi)_m\cap H_KE_m((\cO_K^\times)^{p^{-m}}) = E_m$.

This fact implies that $H_E^\chi\cap H_KE_m((\cO_K^\times)^{p^{-m}})$ is equal to the field $D' := H_E^\chi\cap E_m$.

Then the same approach as above shows that $D' = E$ since the conjugation action of $\Delta$ on $\Gal(D'/E)$ is trivial (as $D'\subseteq E_m$) whilst $\Gal(D'/E)$ also embeds into an $\mathcal{O}$-module upon which $\Delta$ acts via the non-trivial character $\chi$.

We have now shown that $H_E^\chi\cap H_KE_m((\cO_K^\times)^{p^{-m}}) = E$ and so, by Chebotarev's Density Theorem, we can choose a set of primes $\fn$ in $\cN_m$ with the property that
${\rm Cl}(\mathcal{O}_{E,\Sigma_{\fn}})_\chi$ is trivial.

We claim that the cofinal subset $\cN_m'$ of $\cN_m$ comprising multiples of $\fn$ has both of the stated properties (i) and (ii).

To prove this we fix $\fm$ in $\cN_m'$. Then ${\rm Cl}(\mathcal{O}_{E,\Sigma_{\fm}})_\chi$ is trivial and so Proposition \ref{new one}(i), with $\Sigma = S_{\rm ram}(E/K) \cup \Sigma_\fm$ and $T$ empty, implies that $H^0(C_{E,\Sigma})_\chi = (\mathcal{O}^\times_{E,\Sigma})_\chi = (U_{E,\fm})_\chi$ and that $H^1(C_{E,\Sigma})_\chi = X_{E,\Sigma,\chi} = Y_{E,\Sigma_\fm,\chi}$ is a free $\cO[\Pi]$-module of rank $r + \nu(\fm)$. (Here we used Hypothesis (H$_3$).)

Since $(C_{E,\Sigma})_\chi$ belongs to $D^{\rm perf}(\cO[\Pi])$ these facts combine to imply that $(U_{F,\fm})_\chi$ is free of rank $r + \nu(\fm)$, as required to prove claim (i).

Claim (ii) then follows since the vanishing of ${\rm Cl}(\mathcal{O}_{E,\Sigma_{\fm}})_\chi$ implies that the natural inclusion $(U_{E,\fm})_\chi/(p^m) = (U_{E,\fm}/p^m)_\chi\to \cS_m^\fm$ is bijective.
\end{proof}

\subsubsection{}Lemma \ref{partial H map} implies that to construct a surjective homomorphism of $\cO[[\cG_{L\cK}]]$-modules $\Delta^\chi_{F,r}: {\rm VS}_{L\cK}^{(p)} \to {\rm SS}_r(F/K,\chi)$ it is enough to construct a surjective map of $\cO[[\cG_{L\cK}]]$-modules $\Delta^\chi_{F,r,1}: {\rm VS}_{L\cK}^{(p)}\to {\rm SS}^{\rm glob}_r(F/K,\chi)$, and then to define $\Delta^\chi_{F,r}$ to be $\varrho_{F,\chi}\circ\Delta^\chi_{F,r,1}$. To do this we shall mimic the construction of `horizontal determinant systems' from \cite[\S3.3]{sbA}.

For each $\fq$ in $\cP_0$ we write $K(\fq)/K$ for the maximal $p$-extension of $K$ inside its ray class field modulo $\fq$. For each $\fn$ in $\cN_0$ we write $E(\fn)$ for the compositum $E\prod_{\fq \mid \fn} K(\fq)$ and set $S_\fn:= S(E(\fn)) = S_{\rm ram}(E/K) \cup \Sigma_\fn$.

For each $\fn$ in $\cN_0$ we then define $\Delta_\fn$ to be the composite homomorphism
\begin{equation*}\label{delta def} \Det_{\mathcal{G}_{E(\fn)}}(C_{E(\fn),S_\fn})_\chi \to \Det_{\cO[\Pi]}((C_{E,S_\fn})_\chi) \xrightarrow{\Theta_{E,\fn}}{\bigcap}_{\cO[\Pi]}^{r+\nu(\fn)} U_{E,\fn,\chi},\end{equation*}
where the first map is the natural projection and $\Theta_{E,\fn}$ is defined just as in (\ref{rational-projection-map}) but with $S(E)$ replaced by $S_\fn$ and with respect to a fixed choice of $\mathcal{O}[\Pi]$-basis of the free module
$Y_{E,\Sigma_\fn,\chi}$. In particular, the fact that $\Theta_{E,\fn}$ maps $\Det_{\cO[\Pi]}((C_{E,S_\fn})_\chi)$ to ${\bigcap}_{\cO[\Pi]}^{r+\nu(\fn)} U_{E,\fn,\chi}$ follows by an application of \cite[Prop. A.11(ii)]{sbA}, just as in \S\ref{proof tech2}.

We write $\Delta_\fn'$ for the composite of $\Delta_\fn$ and the natural (surjective) projection map from ${\rm VS}_{L\cK}^{(p)}$ to $\Det_{\mathcal{G}_{E(\fn)}}(C_{E(\fn),S_\fn})_\chi$.

Then, provided that one makes a compatible choice of bases of the modules $Y_{E,\Sigma_\fn,\chi}$, these maps $\Delta'_\fn$ are compatible with the transition morphisms  $v_{\fm,\fn}$ as $\fn$ varies over $\cN_0$ and so lead to the construction of a map $\Delta_{F,r,1}^\chi$.

In addition, since the modules $\ker(\Delta'_\fn)$ are compact the derived limit $\varprojlim^1_{\fn \in \cN_0}\ker(\Delta'_\fn)$ vanishes.

To prove that $\Delta_{F,r,1}^\chi$, and hence also $\Delta_{F,r}^\chi$, is surjective it is therefore enough to prove that the maps $\Delta_\fn$ are  surjective for all $\fn$ in the cofinal subset $\cN_0'$ of $\cN_0$ comprising products $\fm$ for which ${\rm Cl}(\mathcal{O}_{E,\Sigma_\fm})_{\chi}$ vanishes.

This is true since for all such $\fn$ the argument of Lemma \ref{cofinal}(i) shows the $\cO[\Pi]$-modules $H^0((C_{E,S_\fn})_\chi) = U_{E,\fn,\chi}$ and $H^1((C_{E,S_\fn})_\chi) = Y_{E,\Sigma_\fn,\chi}$ to be free of rank $r + \nu(\fn)$ and hence that $\Theta_{E,\fn}$ coincides with the natural `passage-to-cohomology' isomorphism

\[ \Det_{\cO[\Pi]}((C_{E,S_\fn})_\chi) \cong \left({\bigwedge}_{\cO[\Pi]}^{r+\nu(\fn)} U_{E,\fn,\chi}\right)\otimes \left({\bigwedge}_{\cO[\Pi]}^{r+\nu(\fn)} Y^\ast_{E,\Sigma_\fn,\chi}\right) \cong {\bigwedge}_{\cO[\Pi]}^{r+\nu(\fn)} U_{E,\fn,\chi} \]
that exists in this case.

\subsubsection{}To complete the proof of Theorem \ref{key diagram} it suffices to check the diagram commutes.

To do this, for each $\fn$ in $\cN_0$ we set $\cG_\fn:=\Gal(E(\fn)/K)$ and $\cH_\fn:=\Gal(F(\fn)/F)$ and write $I_\fn$ for the kernel of the augmentation map $\ZZ[\cH_\fn] \to \ZZ$.

For each product $\fn$ in $\cN_0$ and each prime $\fq$ in $\cP_0$ we define a map $\varphi_\fq^\fn$ as follows:
$$
\varphi_\fq^\fn: E^\times \to \ZZ[\cG_E]\otimes_\ZZ I_\fn/I_\fn^2  ;\ a \mapsto \sum_{\sigma \in  \cG_E}\sigma^{-1} \otimes ({\rm rec}_\fQ(\sigma a)-1),
$$
where $\fQ$ is a fixed place of $E$ lying above $\fq$ and ${\rm rec}_\fQ: E^\times  \to \Gal(E(\fn)/E)\cong \cH_\fn$ is the local reciprocity map at $\fQ$. This map induces a map $
U_{E,\Sigma_\fn,\chi} \to \cO[\Pi] \otimes_\ZZ I_\fn/I_\fn^2$, that we continue to denote by $\varphi_\fq^\fn$.

Upon combining these maps for each prime divisor $\fq$ of $\fn$ we obtain a map of $\cO[\Pi]$-modules
\[ {\bigwedge}_{\fq \mid \fn}\varphi_\fq^\fn :  {\bigcap}_{\cO[\Pi]}^{r+\nu(\fn)} U_{E,\Sigma_\fn,\chi}  \to {\bigcap}_{\cO[\Pi]}^{r} U_{E,\Sigma_\fn,\chi} \otimes I_\fn^{\nu(\fn)}/I_\fn^{\nu(\fn)+1}.\]

Finally we recall (from, for example, \cite[(4.1.1)]{bss3}) that there exists a natural injective homomorphism
\begin{eqnarray}\label{iota}
\iota_\fn: {\bigcap}_{\cO[\Pi]}^{r} U_{E,\Sigma_\fn,\chi} \otimes_\ZZ I_\fn^{\nu(\fn)}/I_\fn^{\nu(\fn)+1} \hookrightarrow {\bigcap}_{\cO[\cG_\fn]}^{r} U_{E(\fn),\Sigma_\fn,\chi} \otimes_\ZZ \ZZ[\cH_\fn]/ I_\fn^{\nu(\fn)+1}.
\end{eqnarray}

Then, after unwinding the definitions of the maps $\Theta_{L\cK}^\chi$ and $\Delta^\chi_{F,r}$, the argument that is used to prove \cite[Th. 4.13]{bss3} shows that the square in Theorem \ref{key diagram} is commutative if and only if for all products $\fn$ in $\cN_0$ and all elements $b_\fn$ of $\Det_{\mathcal{G}_{E(\fn)}}(C_{E(\fn),S_\fn})$ one has
$$\sum_{\sigma \in \cH_n} \sigma \left(\Theta_{E(\fn)}(b_{\fn})^\chi\right) \otimes \sigma^{-1} = \iota_\fn\left(\left({\bigwedge}_{\fq \mid \fn}\varphi_\fq^\fn \right)(\Delta_\fn(b_{\fn}^\chi))\right)$$
in the module ${\bigcap}_{\cO[\cG_\fn]}^r U_{E(\fn),\Sigma_\fn,\chi}\otimes \ZZ[\cH_\fn]/\cI_\fn^{\nu(\fn)+1}$.

To complete the proof of Theorem \ref{key diagram}, it is thus enough to note that latter equality follows directly from the argument used in \cite[\S5.7]{bks1} to prove \cite[Th. 5.16]{bks1}.

\section{The case $K=\QQ$}

In this section we consider the special case $K = \QQ$ and will, in particular, prove Theorems \ref{Q theorem} and \ref{nc thm}.

Throughout we abbreviate the rings $\cR_\QQ$ and $\cR^{\rm s}_\QQ$ and ideals $\cA_\QQ$ and $\cA^{\rm s}_\QQ$ to
$\cR$, $\cR^{\rm s}$, $\cA$ and $\cA^{\rm s}$ respectively.

We also set $S:=S_\infty(\QQ)$ so that for any finite abelian extension $E$ of $\QQ$ the set $S(E)$ comprises the archimedean place of $\QQ$ together with all rational primes that ramify in $E$.

For each natural number $m$ we 
set $\QQ(m) := \QQ(\mu_m)$ and denote its maximal real subfield by $\QQ(m)^+$.

We write $\mathbb{N}^\dagger$ for the subset of $\mathbb{N}\setminus \{1\}$ comprising all numbers that are not congruent to $2$ modulo $4$ and $\mathbb{N}^\ddag$ for the subset $\mathbb{N}^\dagger\setminus \{3,4\}$ of $\mathbb{N}^\dagger$ comprising numbers $m$ for which $\QQ(m)^+\not= \QQ$.

\subsection{}In this subsection we establish some essential preliminary results.

\subsubsection{}The first such result gives an explicit description of the image of the homomorphism $\Theta_\QQ^{\rm s}$ defined in \S\ref{somr section}.

Before stating this result we note the Kronecker-Weber Theorem implies that the set of fields $\{\QQ(m)^+\}_{m\in \mathbb{N}^\ddag}$ is cofinal in the set $\Omega^+:= \Omega(\QQ^{\rm s}/\QQ)$ of non-trivial finite real abelian extensions of $\QQ$ in $\QQ^c$.

The distribution relations satisfied by cyclotomic units therefore imply that there exists a unique system $\varepsilon$ in ${\rm RES}_\QQ$ with the property that
\[ \varepsilon_{\QQ(m)^+} =  (1-\zeta_m)^{(1+\tau)/2}\]
for every $m$ in $\mathbb{N}^\ddag$, 
where $\tau$ denotes complex conjugation.

\begin{theorem}\label{etnc prop} The $\mathcal{R}^{\rm s}$-module ${\rm ES}^{\rm b}_\QQ$ is free of rank one, with basis $\varepsilon$.
\end{theorem}

\begin{proof} For each $E$ in $\Omega^+$ we write $\vartheta_{E}$ for the composite isomorphism of $\RR[\cG_E]$-modules
\begin{align*}
    \RR \otimes_\ZZ \Det_{\cG_E}(C_{E,S(E)}) &\xrightarrow{\sim} \Det_{\RR[\cG_E]}(\RR \otimes_\ZZ \cO_{E,S(E)}^\times) \otimes_{\RR[\cG_E]} \Det_{\RR[\cG_E]}^{-1}(\RR \otimes_\ZZ X_{E,S(E)})\\
    &\xrightarrow{\sim} \Det_{\RR[\cG_E]}(\RR \otimes_\ZZ X_{E,S(E)}) \otimes_{\RR[\cG_E]} \Det_{\RR[\cG_E]}^{-1}(\RR \otimes_\ZZ X_{E,S(E)})\\
    &\xrightarrow{\sim} \RR[\cG_E].
\end{align*}
Here the first arrow is induced by the descriptions in Proposition \ref{new one}(i) and the natural passage-to-cohomology map, the final arrow is the canonical evaluation map and the second arrow is induced by the canonical Dirichlet regulator isomorphism
\begin{equation}\label{dir iso} \lambda_E:\RR \otimes_\ZZ \cO_{E,S(E)}^\times \cong \RR \otimes_\ZZ X_{E,S(E)}\end{equation}
that sends each $u$ in $\cO_{E,S(E)}^\times$ to $-\sum_{w}{\rm log} (|u|_w)\cdot w$, where in the sum $w$ runs over all places of $E$ above those in $S(E)$ and $|-|_w$ denotes the absolute value with respect to $w$ (normalized as in \cite[Chap. 0, 0.2]{tate}).

We then write $\mathfrak{z}_{E}$ for the pre-image under $\vartheta_{E}$ of the element
\[ \theta_{E,S(E)}^*(0) := \sum_{\chi\in \widehat{\cG_E}} L^\ast_{S(E)}(\chi^{-1}, 0)e_\chi\]
of $\RR[\cG_E]^\times$, where $L^\ast_{S(E)}(\chi^{-1}, 0)$ denotes the leading term in the Taylor expansion at $s=0$ of the series $L_{S(E)}(\chi^{-1}, s)$.

Then, by Lemma \ref{etnc basis} below, the collection $\mathfrak{z} := (\mathfrak{z}_{E})_{E \in \Omega^+}$ is an $\cR^{\rm s}$-basis of $\VS(\QQ^{\rm s}/\QQ)$. In addition, by the argument of \cite[Lem. 5.4]{bs}, one knows that the map $\Theta_\QQ^{\rm s}$ is injective.

Hence, the claimed result will follow if we can show for each $m$ in $\mathbb{N}^\ddag$ the element $\Theta_\QQ^{\rm s}(\mathfrak{z})_{\QQ(m)^+} = \Theta_{\QQ(m)^+}(\mathfrak{z}_{\QQ(m)^+})$ is equal to $(1-\zeta_m)^{(1+\tau)/2}$.

To do this we fix such an $m$ and set $S(m) := S(\QQ(m)) = S(\QQ(m)^+)$ and $\cG_m := \cG_{\QQ(m)^+}$. We then fix an embedding $j: \QQ^c \to \CC$ and recall (from, for example \cite[Chap. 3, \S5]{tate}) that for each $\chi$ in $\widehat{\cG_m}$ the first derivative $L_{S(m)}'(\chi, s)$ is holomorphic at $s=0$ and such that
 \[ L_{S(m)}'(\chi, 0) = -\frac{1}{2}\sum_{\sigma \in \mathcal{G}_{m}} \chi(\sigma)\log|(1-\zeta_m^{\sigma})^{1+\tau}|
    \]
where both $\chi(\sigma)$ and $1-\zeta_m^{\sigma}$ are regarded as complex numbers via $j$.

This equality implies, firstly, that the image of $(1-\zeta_m)^{(1+\tau)/2}$ in $\QQ\cdot \mathcal{O}_{\QQ(m)^+,S(m)}^\times$ is stable under the action of the idempotent $e_{(m)} := e_{(\QQ(m)^+)}$ and then, secondly, that its image under the isomorphism (\ref{dir iso}) is equal to $e_{(m)}\cdot \theta_{\QQ(m)^+,S(m)}^*(0)\cdot (w_j-w_0)$, where $w_j$ is the archimedean place of $\QQ(m)^+$ that corresponds to $j$ and $w_0$ is any choice of place of $\QQ(m)^+$ that lies above a prime divisor of $m$.

This latter fact then combines with the explicit definition (\ref{rational-projection-map}) of the map $\Theta_{\QQ(m)^+}$ to imply that %
\begin{equation}\label{zeta cyclo} \Theta_{\QQ(m)^+}(\mathfrak{z}_{\QQ(m)^+}) = e_{(m)}((1-\zeta_m)^{(1+\tau)/2}) = (1-\zeta_m)^{(1+\tau)/2},\end{equation}
as required.
\end{proof}

\begin{lemma}\label{etnc basis} The $\cR^{\rm s}$-module $\VS(\QQ^{\rm s}/\QQ)$ is free of rank one, with basis $(\mathfrak{z}_{E})_{E \in \Omega^+}$.
\end{lemma}

\begin{proof} At the outset, we fix $E$ in $\Omega^+$ and recall (from \cite[Prop. 3.4]{bks1}) that the equivariant Tamagawa Number Conjecture for the pair $(h^0(\Spec(E)), \ZZ[\cG_E])$ asserts that $\Det_{\cG_E}(C_{E,S(E)})$ is a free $\ZZ[\cG_E]$-module with basis $\mathfrak{z}_{E}$.

We further recall that this conjecture is known to be valid by the work of the first named author and Greither in \cite{bg} and of Flach in \cite{flach}.

Given the explicit definition (in Definition \ref{definition vertical}) of $\VS(\QQ^{\rm s}/\QQ)$ as an inverse limit, the claimed result will therefore follow if we can show that for each pair of fields $E$ and $E'$ in $\Omega^+$ with $E \subseteq E'$ one has $\nu_{E'/E}(\mathfrak{z}_{E'}) = \mathfrak{z}_{E}$.

To prove this we use Remark \ref{extra S resolution remark} to identify $\DR\Gamma(\kappa(v)_\mathcal{W}, \ZZ[\cG_E])^\ast[-1]$ for each place $v$ in $S(E')\setminus S(E)$ with the complex $\Psi_v$ that is equal to $\mathbb{Z}[\cG_E]$ in degrees zero and one and has the differential $x\mapsto
(1 - \Fr_v^{-1})x$.

We write $Y_v$ for the free abelian group on the set of places of $E$ above $v$ and, fixing a place $w_v$ of $E$ above $v$,  note there are isomorphisms $\iota_v^i: H^i(\Psi_v) \cong Y_v$ for $i\in \{0,1\}$ with $\iota_v^0(x) = |\cG_{E,v}|^{-1}x\cdot w_v$ and $\iota_v^1(x) = x\cdot w_v$
where, we recall, $\cG_{E,v}$ denotes the decomposition subgroup of $v$ in $\cG_E$.

The key fact now is that the $\Gal(E'/E)$-invariants of $\vartheta_{E'}$ differs from the composite $\vartheta_E\circ\nu_{E'/E}$ only in that for each $v \in S(E')\setminus S(E)$ and each $\chi$ in $\widehat{\cG_E}$ these maps respectively use the upper and lower composite homomorphisms in the following diagram

\[\begin{CD}
(\Det_{\mathcal{G}_{E}}(\Psi_v))_\chi @> \alpha_1 >> \Det_{\CC}(Y_{v,\chi})\cdot\Det_{\CC}(Y_{v,\chi})^{-1} @> \alpha_2 >> \CC\\
@\vert @. @VV \cdot \epsilon^\chi_v V\\
  (\Det_{\mathcal{G}_{E}}(\Psi_v))_\chi @> \alpha_3 >> \Det_{\CC}(\ZZ[\cG_E]_\chi)\cdot\Det_{\CC}(\ZZ[\cG_E]_\chi)^{-1} @>\alpha_4 >> \CC.\end{CD}
\]
Here $\alpha_1$ denotes the morphism induced by the maps $\iota_v^0$ and $\iota_v^1$, $\alpha_2$ the morphism induced by multiplication by ${\rm log}({\rm N}(w_v))$, $\alpha_3$ is the obvious identification, $\alpha_4$ is induced by the identity map on $\ZZ[\cG_E]_\chi$ and $\epsilon^\chi_v$ is defined to be $1 - \chi(\Fr_v)^{-1}$ if $\chi(\Fr_v) \not= 1$ and to be $|\cG_{E,v}|^{-1}\cdot {\rm log}({\rm N}(w_v)) = {\rm log}({\rm N}(v))$ otherwise.

The claimed result then follows from the fact that the argument of \cite[Lem. 10]{bf} shows that the above diagram commutes, whilst an explicit computation using (\ref{changeS}) shows that for each $\chi$ in $\widehat{\cG_E}$ one has
\[ L_{S(E')}^*(\chi,0) = \left(\prod_{v \in S(E')\setminus S(E)}\epsilon_v^\chi\right)\cdot L_{S(E)}^*(\chi,0).\]
%
\end{proof}

\subsubsection{}In the next two results we establish some useful properties of the ideal $\mathcal{A}^{\rm s}$. We write $\mathcal{I}_{\QQ,(2)}$ for the kernel of the natural projection map $\cR^{\rm s} \to \ZZ/(2)$.

\begin{lemma}\label{tor ann lem} One has $\{r\in \mathcal{R}^{\rm s}\mid r({\rm ES}^{\rm b}_\QQ) \subseteq {\rm ES}_\QQ\} = \mathcal{I}_{\QQ,(2)} = \mathcal{A}^{\rm s}$.
\end{lemma}

\begin{proof} We set $X:= \{r\in \mathcal{R}^{\rm s}\mid r({\rm ES}^{\rm b}_\QQ) \subseteq {\rm ES}_\QQ\}$.

We also recall that the equality  $\mathcal{I}_{\QQ,(2)} = \mathcal{A}^{\rm s}$ was explained just before the statement of Theorem \ref{Q theorem}  and that Theorem \ref{tech req}(ii) implies $X$ contains $\mathcal{A}^{\rm s}$.

On the other hand, if $r$ belongs to $\mathcal{R}^{\rm s}\setminus\mathcal{I}_{\QQ,(2)}$, then it can be written as $r = r' + 1$ with $r'\in \mathcal{I}_{\QQ,(2)}$.

To deduce $\mathcal{I}_{\QQ,(2)} = X$ it is thus enough to show that the system $\varepsilon$ that occurs in Theorem  \ref{etnc prop} does not belong to ${\rm ES}_\QQ$, or equivalently that there exists an $m$ in $\mathbb{N}^\ddag$ for which $\varepsilon_{\QQ(m)^+}$ is not contained in $\mathcal{O}^\times_{\QQ(m)^+,S(m),{\rm tf}}$.

But, if $\varepsilon_{\QQ(m)^+}$ belongs to $\mathcal{O}^\times_{\QQ(m)^+,S(m),{\rm tf}}$, then there exists an element $u$ of $\QQ(m)^+$ with $\varepsilon_{\QQ(m)^+} = \pm u$ and hence also
\[ u^2 = (\varepsilon_{\QQ(m)^+})^2 = (1-\zeta_m)^{1+\tau} = -\zeta^{-1}_m(1-\zeta_m)^2. \]
This implies $-\zeta_m$ is a square in $\QQ(m)$ and so is impossible if $m$ is divisible by $4$.
\end{proof}

In the sequel we write $\mu_\infty$ for the union of $\mu_m$ over all $m$, set $\mu_\infty^* := \mu_\infty\setminus \{1\}$ and write $\mathcal{F}$ for the multiplicative group of $G_\QQ$-equivariant maps from $\mu_\infty^*$ to $\QQ^{c,\times}$.

We note that $\mathcal{F}$ is naturally an $\mathcal{R}$-module and that it contains the function $\Phi$ that sends each element $\zeta$ of $\mu_\infty^\ast$ to $1-\zeta$.

\begin{lemma}\label{coho comp} The image of the set $\{r \in \mathcal{R}\mid (1-\zeta_m)^{r(\tau-1)} = 1 \,\text{ for all }\, m\}$ under the projection map $\mathcal{R} \to \mathcal{R}^{\rm s}$ is $\mathcal{A}^{\rm s}$.
\end{lemma}

\begin{proof} Write $\widehat{\ZZ}(1)$ for the inverse limit of the groups $\mu_m$ with respect to the transition morphisms $\mu_m\to \mu_{m'}$ for each divisor $m'$ of $m$ that are given by raising to the power $m/m'$.

Then $\widehat{\ZZ}(1)$ is naturally an $\mathcal{R}$-module and the result of \cite[Th. 1.2]{bs} proves that it lies in an exact sequence of $\mathcal{R}$-modules
\begin{equation*}\label{th 1.2 seq} 0 \to \widehat{\ZZ}(1) \to \mathcal{R}\cdot \Phi \xrightarrow{\pi} \mathcal{R}^{\rm s} \to 0\end{equation*}
in which $\pi$ sends each element $r\cdot \Phi$ to the projection of $r$ in $\mathcal{R}^{\rm s}$.

Set $\Gamma = \Gal(\CC/\RR) = \langle\tau\rangle$. Then, since $H^0(\Gamma, \widehat{\ZZ}(1)) $ vanishes, the long exact sequence of Tate cohomology of this sequence gives an exact sequence
\begin{equation}\label{second seq} 0 \to \hat H^0(\Gamma, \mathcal{R}\cdot \Phi) \to \hat H^0(\Gamma, \mathcal{R}^{\rm s}) \xrightarrow{\pi'} \hat H^{-1}(\Gamma,\widehat{\ZZ}(1)).\end{equation}
Here the map $\pi'$ is the composite of the connecting homomorphism
\[ \hat H^0(\Gamma, \mathcal{R}^{\rm s}) \to \hat H^{1}(\Gamma,\widehat{\ZZ}(1))\]
that is induced by the above exact sequence and the canonical isomorphism $\hat H^1(\Gamma,\widehat{\ZZ}(1)) \cong \hat H^{-1}(\Gamma,\widehat{\ZZ}(1))$ that results from the fact that Tate cohomology over $\Gamma$ is periodic of order $2$.

Now it is clear that the groups $\hat H^0(\Gamma, \mathcal{R}^{\rm s})$ and $\hat H^{-1}(\Gamma,\widehat{\ZZ}(1))$ respectively identify with
$\mathcal{R}^{\rm s}/(2\cdot \mathcal{R}^{\rm s})$ and $\ZZ/(2\cdot \ZZ)$ and an explicit computation shows that, with respect to these identifications, the map $\pi'$ in (\ref{second seq}) is induced by the natural projection map $\mathcal{R}^{\rm s} \to \ZZ$.


Given these facts, the exact sequence (\ref{second seq}) implies that the image under $\pi$ of
 $H^0(\Gamma,\mathcal{R}\cdot \Phi)$ is equal to $\mathcal{I}_{\QQ,(2)}$, and hence also to $\mathcal{A}^{\rm s}$ by Lemma \ref{tor ann lem}.

 It therefore suffices to note that an element $r$ of $\mathcal{R}$ is such that $r\cdot\Phi$ belongs to $H^0(\Gamma,\mathcal{R}\cdot \Phi)$ if and only if one has $(1-\zeta_m)^{r(\tau-1)} = 1$ for all $m$.
\end{proof}

\subsection{Circular distributions and the proof of Theorem \ref{Q theorem}}\label{Q thm sec} In this section we quickly review Coleman's Conjecture concerning circular distributions and then prove Theorem \ref{Q theorem}.

\subsubsection{Coleman's Conjecture} In the 1980s Coleman defined a `circular distribution', or `distribution' for short in the sequel, to be a function $f$ in $\calF$  that satisfies the relation
\begin{equation*}\label{defining prop} \prod_{\zeta^a = \varepsilon} f(\zeta) = f(\varepsilon)\end{equation*}
for all natural numbers $a$ and all $\varepsilon$ in $\mu_\infty^*$. (A similar notion was also subsequently introduced by Coates in \cite{coates} in the context of abelian extensions of imaginary quadratic fields.)

It is clear that the subset $\calF^{\rm d}$ of $\calF$ comprising all such distributions is naturally an $\mathcal{R}$-module.

Further, recalling the module ${\rm ES}_\QQ^{\rm cl} = {\rm ES}^{\rm cl}(\QQ^{\rm ab}/\QQ)$ of classical Euler systems of rank one for $\mathbb{G}_m$ over $\QQ^{\rm ab}/\QQ$, as discussed in Remark \ref{not classical}, there exists a canonical isomorphism of $\mathcal{R}$-modules
\begin{equation}\label{es iso} \calF^{\rm d} \cong {\rm ES}_\QQ^{\rm cl}.\end{equation}
This map sends each $f$ in $\calF^{\rm d}$ to the unique element $c_f$ of ${\rm ES}_\QQ^{\rm cl}$ with $c_{f,\QQ(m)} = f(\zeta_m)$ for all $m\in \mathbb{N}^\dagger$ and its inverse sends each $c$ in ${\rm ES}_\QQ^{\rm cl}$ to the unique function $f_c$ in $\mathcal{F}^{\rm d}$ that satisfies
\[ f_c(\zeta_m) := \begin{cases} c_{\QQ(m)}, &\text{ if $m\in \mathbb{N}^\dagger$}\\
                                 (c_{\QQ(m')})^{1-{\rm Fr}_2^{-1}},  &\text{ if $m= 2m'$ with $m'$ odd and $m' > 1$,}\\
                                  {\rm N}_{\QQ(4)/\QQ}(c_{\QQ(4)}), &\text{ if $m= 2$.}\end{cases}\]
(In this regard, we point out that our definition of the set of fields $\Omega(\QQ^c/\QQ)$ means that it does not contain $\QQ$ and hence that a system in ${\rm ES}_\QQ^{\rm cl}$ has no component at the field $\QQ$.)

The function $\Phi$ that is defined just prior to Lemma \ref{coho comp} clearly belongs to $\mathcal{F}^{\rm d}$ and Coleman's Conjecture predicts  that
\begin{equation}\label{cc eq} \calF^{\rm d} = \calF^{\rm d}_{\rm tor} + \mathcal{R}\cdot \Phi.\end{equation}

Since \cite[Th. B]{Seo3} implies that every distribution of finite order has order dividing two, the isomorphism (\ref{es iso}) implies this conjecture is equivalent to asserting that for every system $c = (c_F)_F$ in ${\rm ES}_\QQ^{\rm cl}$ there exists an element $r'_c$ of $\mathcal{R}$ such that for every $m$ in $\mathbb{N}^\dagger$ one has
\begin{equation}\label{coleman} c_{\QQ(m)} = \pm \Phi(\zeta_m)^{r_c'} = \pm (1-\zeta_m)^{r'_c} \end{equation}
in $\QQ(m)^\times$. 

We further recall that the recent result of \cite[Th. 1.2]{bs} implies both that the quotient group $\calF^{\rm d}/(\calF^{\rm d}_{\rm tor} + \mathcal{R}\cdot \Phi)$ is torsion-free and also that the homomorphism of $\mathcal{R}$-modules
\[ \calF^{\rm d}/(\calF^{\rm d}_{\rm tor} + \mathcal{R}\cdot \Phi) \to (1+\tau)\mathcal{F}^{\rm d}/(\mathcal{R}\cdot \Phi^{1+\tau})\]
that sends the class of each $f$ in $\mathcal{F}^{\rm d}$ to the class of $f^{1+\tau}$ is bijective.

In addition, if $f$ belongs to  $\mathcal{F}^{\rm d}$, then \cite[Th. 5.1 and Rem. 5.2]{bs} implies that to show $f^{1+\tau}$ belongs to $\mathcal{R}\cdot \Phi^{1+\tau}$ it is enough to show the existence of a natural number $t$ in $\mathbb{N}^\ast$ and an element $r = r_{f}$ of $\mathcal{R}$ for which one has $f(\zeta_m)^{1+\tau} = (1-\zeta_m)^{(1+\tau)r}$ for all $m$ in $\mathbb{N}^\dagger$ that are divisible by $t$.

These facts combine to imply that the conjectural equality (\ref{cc eq}) is valid if for each system $c$ in ${\rm ES}_\QQ^{\rm cl}$ there exist elements $t_1$ and $t_2$ in $\mathbb{N}^\ast$ and an element $r= r_{c,t_1}$ of $\mathcal{R}$ for which one has
\begin{equation}\label{explicit coleman} (c_{\QQ(m)})^{t_1(1+\tau)} = (1-\zeta_m)^{(1+\tau)r}\end{equation}
for all $m$ in $\mathbb{N}$ that are divisible by $t_2$.
%

\begin{remark}  For more details concerning Coleman's Conjecture, the reader can consult \cite{bs} or the earlier articles \cite{Seo1, Seo2, Seo3, Seo4} of the third author and the associated work of Saikia in \cite{saikia}.\end{remark}

\subsubsection{The proof of Theorem \ref{Q theorem}}\label{proof Q thm} We first assume the validity of Coleman's Conjecture and use it to deduce the validity of Conjecture \ref{main conj} in the case $K = \QQ$.

To do this we fix an element $a$ of $\mathcal{A}^{\rm s}$ and a system $c = (c_F)_F$ in ${\rm ES}_\QQ = {\rm ES}_1(\QQ^{\rm s}/\QQ)$. Then, by the discussion in Remark \ref{not classical}, there exists a canonical system $c'$ in ${\rm ES}^{\rm cl}(\QQ^{\rm s}/\QQ)$ that projects to $c^a$ in ${\rm ES}_\QQ$.

For this system $c'$ one has $c'_{\QQ(m)^+} >0$ for every $m$ in $\mathbb{N}^\ddag$. In particular, since the argument of \cite[Lem. 2.2]{Seo1} implies $c'_{\QQ(m)^+}$ is a global unit for every $m$ in $\mathbb{N}^\ddag$ that is divisible by two distinct primes, one has ${\rm N}^{\QQ(m)^+}_{\QQ}(c'_{\QQ(m)^+})= 1$ for all such $m$.

As a result, there exists a unique system $c''$ in ${\rm ES}^{\rm cl}_\QQ$ for which at each $m$ in $\mathbb{N}^\dagger$ one has
\[ c''_{\QQ(m)} := \begin{cases} c'_{\QQ(m)^+}, &\text{ if $m \in \mathbb{N}^\ddag$,}\\
                                 {\rm N}^{\QQ(9)^+}_{\QQ}(c_{\QQ(9)^+}), &\text{ if $m = 3$,}\\
                                 {\rm N}^{\QQ(8)^+}_{\QQ}(c_{\QQ(8)^+}), &\text{ if $m = 4$.}\end{cases}\]

For this system one has $(c'')^{\tau-1} = 1$. Thus, if Coleman's Conjecture (in the form of (\ref{coleman})) is valid, then Lemma \ref{coho comp} implies the existence of an element $r$ of $\mathcal{R}$ which projects to $\mathcal{R}^{\rm s}$ to give an element $a'$ of $\mathcal{A}^{\rm s}$ and is such that for every $m$ in $\mathbb{N}^\ddag$ one has
\[ c'_{\QQ(m)^+} = c''_{\QQ(m)} = \pm (1-\zeta_m)^{r}\]
and hence also
\[ (c'_{\QQ(m)^+})^{1+\tau} = (1-\zeta_m)^{(1+\tau)r}= (1-\zeta_m)^{(1+\tau)a'} = (\varepsilon_{\QQ(m)^+})^{2a'}.\]

But then in ${\rm RES}_\QQ$ one has
\[ c^a = c' = ((c')^2)^{1/2} = ((c')^{1+\tau})^{1/2} = (\varepsilon^{2a'})^{1/2}  = \varepsilon^{a'}. \]
Since $a$ is an arbitrary element of $\mathcal{A}^{\rm s}$, and $a'$ belongs to $\mathcal{A}^{\rm s}$, this proves the claimed inclusion $\mathcal{A}^{\rm s}\cdot {\rm ES}_\QQ \subseteq \mathcal{A}^{\rm s}\cdot {\rm ES}^{\rm b}_\QQ$.

To prove the converse it suffices to show the latter inclusion implies that every system in ${\rm ES}_\QQ^{\rm cl}$ satisfies the condition (\ref{explicit coleman}).

To do this we fix $c$ in ${\rm ES}_\QQ^{\rm cl}$. Then by first restricting $c$ to the subset $\Omega^+$ of $\Omega(\QQ^{\rm ab}/\QQ)$ and then reducing its values modulo torsion one obtains a system $\tilde c$ in ${\rm ES}_\QQ$.

Further, since $2$ belongs to $\mathcal{A}^{\rm s}$, the assumed inclusion $\mathcal{A}^{\rm s}\cdot {\rm ES}_\QQ \subseteq \mathcal{A}^{\rm s}\cdot {\rm ES}^{\rm b}_\QQ$ combines with Theorem \ref{etnc prop} to imply the existence of an element $r$ of $\mathcal{R}$ such that $\tilde c^4 = \varepsilon^{2r}$ in ${\rm RES}_\QQ$.

This equality of functions implies that for each $m$ in $\mathbb{N}^\ddag$ one has
\[ (c_{\QQ(m)})^{4(1+\tau)} = (\tilde c_{\QQ(m)^+})^4 = (\varepsilon_{\QQ(m)^+})^{2r} = (1-\zeta_m)^{(1+\tau)r}.\]
This shows that $c$ verifies the condition (\ref{explicit coleman}) with $t_1 = t_2 = 4$ and hence implies Coleman's Conjecture.

%

This completes  the proof of Theorem \ref{Q theorem}.

\subsection{Norm coherent sequences in $p$-power conductor cyclotomic fields}\label{nc sec}

In this final section we shall give a proof of Theorem \ref{nc thm}. To do this we fix an odd prime $p$ and for each natural number $n$ set $G_n := \Gal(\QQ(p^n)/\QQ)$ and $G_n^+ := \Gal(\QQ(p^n)^+/\QQ)$.

At the outset we recall that the discussion in \cite[\S3]{Seo4} shows that if $(a_n)_n$ is a norm coherent sequence in $\bigcup_n\QQ(p^n)$ such that each $a_n$ belongs to $\ZZ[G_n]\cdot(1-\zeta_{p^n})$, then there exists a circular distribution $f$ such that $f(\zeta_{p^n}) = a_{n}$ for all $n$.

To prove Theorem \ref{nc thm}, it is thus enough to fix $f$ in $\mathcal{F}^{\rm d}$ and show that there exists a natural number $t$ (possibly depending on $f$) for which $f(\zeta_{p^n})^{2^t}$ belongs to $\ZZ[G_n]\cdot (1-\zeta_{p^n})$ for all $n$.

Hence, since \cite[Th. B]{Seo2} implies the existence of a natural number $c$ for which $f(\zeta_{p^n})^{c}$ belongs to $\ZZ[G_n]\cdot (1-\zeta_{p^n})$ for all $n$, it is enough to show for each $n$ that there exists a natural number $t_n$ with $f(\zeta_{p^n})^{2^{t_n}}\in \ZZ[G_n]\cdot (1-\zeta_{p^n})$.

In addition, for any $t$ in $\mathbb{N}^\ast$ one has $f^{2^t} = f^{(1-\tau)2^{t-1}}f^{(1+\tau)2^{t-1}}$ and, by \cite[Th. 4.1]{bs}, one knows that
 $f(\zeta_{p^n})^{1-\tau}\in \ZZ[G_n]\cdot (1-\zeta_{p^n})$ for all $n$.

It therefore suffices to show that for each $n$, and for every odd prime $\ell$, one has
\[ f(\zeta_{p^n})^{1+\tau}\in \ZZ_\ell[G_n]\cdot \varepsilon_{n},\]
where we set
\[ \varepsilon_{n} := (1-\zeta_{p^n})^{1+\tau}.\]

For the prime $\ell=p$ the above containment is proved in \cite[Th. 3.1]{bs} and so we assume in the sequel that $\ell\not= p$.

We decompose $G_n$ as a direct product $\Lambda_n\times \Delta_n$ with $\Lambda_n$ the Sylow $\ell$-subgroup of $G_n$. We set $E := \QQ(p^n)$ and $E^+ := \QQ(p^n)^+$, write $F$ and $L$ for the fixed fields of $E$ by $\Delta_n$ and $\Lambda_n$ respectively and write $L^+$ for the maximal real subfield $L \cap E^+$ of $L$.

Then it suffices to prove that for each homomorphism $\phi: \Delta_n \to \QQ_\ell^{c,\times}$ that is trivial on $\tau$ one has %
\begin{equation}\label{required cont} f(\zeta_{p^n})^\phi\in \ZZ[G_n]_\phi\cdot \varepsilon_{n},\end{equation}
where we use the same notation for $\phi$-components as in \S\ref{H hyp sec}.

We assume first that $\phi$ is trivial. In this case we can identify $X_\phi$ for each $\Delta_n$-module $X$ as a submodule of $\ZZ_\ell\otimes_\ZZ X$ by sending each element $x^\phi$ to $\sum_{h \in \Delta_n}h(x)$.

Then, with respect to this identification, one has $\ZZ[G_n]_\phi\cdot \varepsilon_{n} = \ZZ_\ell[G_n]\cdot {\rm N}_{E/F}(\varepsilon_{n})$. In addition, since $F/\QQ$ is an $\ell$-power degree extension that is only ramified at the prime $p$, one knows that the class number of $F$ is not divisible by $\ell$ (by \cite[Th. 10.4]{wash}).

Upon combining the latter fact with the formulas proved by Sinnott in \cite[Th. 4.1 and Th. 5.1]{sinnott} one finds that ${\rm N}_{E/F}(\varepsilon_{n})$ generates over $\ZZ_\ell[G_n]$ the $\ell$-completion of the full group of $p$-units of $F$.

The required containment (\ref{required cont}) is thus true in this case since $f(\zeta_{p^n})^\phi = {\rm N}_{E/F}(f(\zeta_{p^n}))$ is a $p$-unit in $F$ (by \cite[Lem. 2.2]{Seo1}).

We now assume $\phi$ is non-trivial and regard it as a homomorphism $\Gal(L^+/\QQ) \to \QQ_\ell^{c,\times}$. We also note that 
 $L^+$ is not contained in $\QQ(\mu_\ell)$, that $\phi^2$ is not equal to the Teichm\"uller character at $\ell$ if $\ell = 3$, that $p$ is not split in $L^+$ and that the archimedean place of $\QQ$ splits in $L^+$.

This shows that the pair $F$ and $\phi$ satisfies all of the hypotheses (H$_1$), (H$_2$), (H$_3$), (H$_4$) and (H$_5$) (with $p$ replaced by $\ell$) that are listed in \S\ref{H hyp sec}.

We may therefore apply Theorem \ref{koly thm} to the system $\tilde c_f$ in ${\rm ES}_\QQ$ that is obtained from $f$ by the method used in the latter part of \S\ref{proof Q thm}.

Noting that $\mathcal{A}_{E,\phi} = \ZZ[G_n]_\phi$ (since $\ell$ is odd and $\phi$ is even), we deduce in this way that the element
 $(f(\zeta_{p^n})^\phi)^2 = (f(\zeta_{p^n})^{1+\tau})^\phi = (\tilde c_{f,E^+})^\phi$ belongs to $\Theta_{E^+}(\Det_{G^+_n}(C_{E^+,S(E)}))_\phi$.

The required containment (\ref{required cont}) is therefore true in this case since the equality (\ref{zeta cyclo}) (with $m = p^n$) combines with the argument of Lemma \ref{etnc basis} to imply that $\Theta_{E^+}(\Det_{G^+_n}(C_{E^+,S(E)}))_\phi = \ZZ[G^+_n]_\phi\cdot \varepsilon_{n}$.

This completes the proof of Theorem \ref{nc thm}.

\begin{remark} The result of Theorem \ref{nc thm} does not extend in a straightforward way to more general towers of cylotomic fields. To be specific, if $p$ is an odd prime and $m$ is a multiple of $p$ that is not a power of $p$, then there can exist norm-compatible sequences of cyclotomic units in the tower $\bigcup_{n}\QQ(mp^n)$ that are not equal to the restriction of any circular distribution. For an explicit example of this phenomenon see \cite[Lem. 2.10]{bs}. \end{remark}

\end{document}